\documentclass[12pt]{amsart}
\pdfoutput=1
\usepackage{heuristica}
\usepackage[heuristica,vvarbb,bigdelims]{newtxmath}

\usepackage{fullpage,url,mathrsfs,amssymb,color,multicol,enumerate,amscd}
\usepackage{multirow}
\usepackage{mathtools}
\usepackage[all]{xy} 
\usepackage{verbatim,moreverb}

\usepackage[OT2,T1]{fontenc}

\usepackage{tikz-cd}
\usepackage{xfrac}
\usepackage{float}

\definecolor{backgroundcolor}{rgb}{0.9,0.8,0.8}

\numberwithin{equation}{section}
\usepackage{array}
\newcolumntype{H}{>{\setbox0=\hbox\bgroup}c<{\egroup}@{}}

\usepackage{setspace}
\usepackage{color}
\newcommand{\vanish}[1]{}

\def\bbar#1{\setbox0=\hbox{$#1$}\dimen0=.2\ht0 \kern\dimen0 }
\newcommand{\defi}[1]{\textsf{#1}} 
\usepackage{enumitem}

\newenvironment{romanenum}{\hfill \begin{enumerate}[label=({\roman*})]
}
{\end{enumerate}}
\DeclareSymbolFont{cyrletters}{OT2}{wncyr}{m}{n}
\DeclareMathSymbol{\Sha}{\mathalpha}{cyrletters}{"58}

\newcommand{\G}{{\mathbb G}}

\newcommand{\PP}{{\mathbb P}}
\newcommand{\Q}{{\mathbb Q}}
\newcommand{\RR}{{\mathbb R}}
\newcommand{\ZZ}{{\mathbb Z}}
\newcommand{\Zhat}{{\hat{\ZZ}}}

\def\bbar#1{\setbox0=\hbox{$#1$}\dimen0=.2\ht0 \kern\dimen0 \overline{\kern-\dimen0 #1}}
\newcommand{\Qbar}{{\overline{\mathbb Q}}} 
\newcommand{\Qab}{{\mathbb Q}^{\operatorname{ab}}}

\newcommand{\frakA}{{\mathfrak A}}
\newcommand{\frakB}{{\mathfrak B}}
\newcommand{\frakG}{{\mathfrak G}}

\newcommand{\calA}{{\mathcal A}}

\newcommand{\calF}{{\mathcal F}}

\newcommand{\calH}{{\mathcal H}}

\DeclareMathOperator{\Aut}{Aut}
\DeclareMathOperator{\Gal}{Gal}
\DeclareMathOperator{\divv}{div}

\newcommand{\GalQ}{{\Gal}(\Qbar/\Q)}

\newcommand{\GL}{\operatorname{GL}}
\newcommand{\SL}{\operatorname{SL}}
\newcommand{\PGL}{\operatorname{PGL}}

\newcommand{\cyc}{\operatorname{cyc}}

\newcommand{\injects}{\hookrightarrow}
\newcommand{\isom}{\simeq}

\newcommand{\intersect}{\cap} 
\newcommand{\union}{\cup} 
\newcommand{\tensor}{\otimes}

\def\CC{\mathbb C}

\usepackage{amsthm}

\definecolor{webcolor}{rgb}{0,0,1}
\definecolor{webbrown}{rgb}{.6,0,0}
\usepackage[
        colorlinks,
        linkcolor=webbrown,  filecolor=webbrown,  citecolor=webbrown, 
        backref,
        pdfauthor={Rakvi}, 
        pdftitle={A Classification of Genus 0 Modular Curves with a Rational Point},
]{hyperref}
\usepackage{cleveref}
\showboxdepth=\maxdimen
\showboxbreadth=\maxdimen
\newtheorem{theorem}{Theorem}[section]
\newtheorem{lemma}[theorem]{Lemma}

\newtheorem{proposition}[theorem]{Proposition}

\theoremstyle{definition}
\newtheorem{definition}[theorem]{Definition}

\newtheorem*{Mazur's Program B}{Mazur's Program B}

\newtheorem{example}{Example}

\theoremstyle{remark}
\newtheorem*{remark}{Remark}

\crefname{section}{§}{§§}
\crefname{lemma}{Lemma}{Lemmas}
\crefname{equation}{equation}{equations}
\crefname{theorem}{Theorem}{Theorems}
\crefname{proposition}{Proposition}{Propositions}

\usepackage[alphabetic,backrefs,lite]{amsrefs} 
\usepackage{datetime}
\date{\today}
\usepackage{mathtools}
\usepackage{amsfonts}
\usepackage{hyperref}
\usepackage{microtype}
\begin{document}

\title[]{A Classification of Genus 0 Modular Curves with rational points}

\author{Rakvi}
\address{Department of Mathematics, University of Pennsylvania, Philadelphia, PA 19104 USA}
\email{rakvi@sas.upenn.edu}

\begin{abstract}

Let $E$ be a non-CM elliptic curve defined over $\Q$. Fix an algebraic closure $\Qbar$ of $\Q$. We get a Galois representation \[\rho_E \colon \GalQ \to \GL_2(\Zhat)\] associated to $E$ by choosing a compatible bases for the $N$-torsion subgroups of $E(\Qbar).$ Associated to an open subgroup $G$ of $\GL_2(\Zhat)$ satisfying $-I \in G$ and $\det(G)=\Zhat^{\times}$, we have the modular curve $(X_G,\pi_G)$ over $\Q$ which loosely parametrises elliptic curves $E$ such that the image of $\rho_E$ is conjugate to a subgroup of $G^t.$ In this article we give a complete classification of all such genus $0$ modular curves that have a rational point. This classification is given in finitely many families. Moreover, each such modular curve can be explicitly computed.

\end{abstract}

\maketitle
\setcounter{tocdepth}{1}

\section{Introduction}

Let $E$ be a non-CM elliptic curve defined over $\Q$. Fix an algebraic closure $\overline{\Q}$ of $\Q$. For a positive integer $N$, let $E[N]$ denote the $N$-torsion subgroup of $E(\overline{\Q}).$ As a $\ZZ/N\ZZ$-module $E[N]$ is free of rank 2. Choosing a basis of $E[N]$, the natural action of $\GalQ$ on $E[N]$ gives a representation,
\[ \rho_{E,N}\colon \GalQ \to \Aut(E[N]) \simeq \GL_{2}(\ZZ/N \ZZ). \] By choosing compatible bases for $E[N]$ with $N \ge 1$, we get a representation
\[ \rho_{E}\colon \GalQ \to \GL_{2}(\hat{\ZZ}) \] such that the composition of $\rho_E$ with the reduction modulo $N$ map $\GL_{2}(\hat{\ZZ}) \rightarrow \GL_{2}(\ZZ/N \ZZ)$ is $\rho_{E,N}.$

The images of $\rho_E$ and $\rho_{E,N}$ are uniquely determined up to conjugacy in $\GL_{2}(\hat{\ZZ})$ and $\GL_{2}(\ZZ/N \ZZ)$, respectively. Serre's open image theorem \cite{MR0387283} says that the group $\rho_{E}(\GalQ)$ is open and hence of finite index in $\GL_{2}(\hat{\ZZ})$. We have $\det(\rho_{E}(\GalQ))=\hat{\ZZ}^{\times}$ since $\det \circ \rho_{E}$ agrees with the cyclotomic character $\chi_{\cyc} \colon \GalQ \to \hat{\ZZ}^{\times}.$ Recall that $\chi_{\cyc}$ satisfies the property $\sigma(\zeta)=\zeta^{\chi_{\cyc}(\sigma)\mod n}$ for any integer $n \ge 1$, $n$-th root of unity $\zeta \in \Qbar$ and $\sigma \in \GalQ.$ \\
The following is a difficult open problem in number theory that serves as a motivation for the classification problems related to modular curves, see section $1$ of \cite{MR0450283}.

\begin{Mazur's Program B}
Given an open subgroup $G$ of $\GL_{2}(\hat{\ZZ})$, classify all the elliptic curves $E$ defined over $\Q$ such that the image of $\rho_E$ is conjugate in $\GL_{2}(\hat{\ZZ})$ to a subgroup of $G$.
\end{Mazur's Program B}

Let $G$ be an open subgroup of $\GL_{2}(\hat{\ZZ})$ satisfying $\det(G)=\hat{\ZZ}^{\times}$ and $-I \in G.$ Associated to $G$, there is a modular curve $X_{G}$ which is a smooth, projective and geometrically connected curve defined over $\Q.$ See \cref{sec:Modular Curves} for background on modular curves. The curve $X_{G}$ loosely parametrises elliptic curves $E$ for which the image of $\rho_{E}$ is conjugate in $\GL_{2}(\hat{\ZZ})$ to a subgroup of $G^{t}$ (here $G^{t}$ is the transpose of $G$). 

\begin{remark}
The choice of $G$ or $G^t$ is not consistent across literature. We consider the natural right action of the matrix group on field of modular functions (see \cref{action on F_N}), therefore we have $G^t$ in the moduli property. Observe that $G \intersect \SL_2(\ZZ)$ and $G^t \intersect \SL_2(\ZZ)$ are conjugate in $\SL_2(\ZZ)$ by the element $(\begin{smallmatrix}
0 & 1 \\ 
-1  & 0
\end{smallmatrix}).$ Therefore the modular curves $X_G$ and $X_{G^t}$ are $\Qab$-twists (see \cref{K-twist} for definition of twists) of each other.  
\end{remark}

An inclusion $G \subseteq G^{'}$ induces a natural morphism $X_{G} \to X_{G^{'}}$. In particular, every $X_{G}$ has an associated morphism $\pi_{G} \colon X_{G} \to X_{\GL_{2}(\Zhat)} \isom\ \PP^{1}_{\Q}$. Throughout this article, we will use the notation $(X_{G},\pi_G)$ to denote a modular curve with its morphism to the $j$-line. 

In this article we provide a complete classification of all genus $0$ modular curves $(X_G,\pi_G)$ where $X_G$ is $\PP^1_{\Q}.$ The code related to this paper can be accessed at \url{https://github.com/Rakvi6893/Classification-of-genus-0-Modular-Curves-that-have-a-Rational-Point}.

Let $X$ be a nice curve over $\Q$ with a nonconstant morphism $\pi_X \colon X \to \mathbb{P}^{1}_{\Q}.$ Let $K$ be a Galois extension of $\Q.$

\begin{definition} \label{K-twist}

A \defi{$K$-twist} of $(X,\pi_X)$ is a pair $(Y,\pi_Y)$ where $Y$ is a curve over $\Q$ and $\pi_Y$ is a morphism $Y \to \mathbb{P}^{1}_{\Q}$ such that there exists an isomorphism $f \colon X_K \to Y_K$ which satisfies $\pi_{Y}\circ f = \pi_{X}.$ \label{isomorphic}We say that $(X,\pi_X)$ and $(Y,\pi_Y)$ are \defi{isomorphic} if $(Y,\pi_Y)$ is a $\Q$-twist of $(X,\pi_X).$ 
\end{definition}

\label{Aut} Let $\Aut_K(X,\pi_X)$ be the subgroup of $\Aut(X_K)$ consisting of $f$ such that $\pi_{X} \circ f= \pi_{X}.$ For $K=\Q$, we will use the notation $\Aut(X,\pi_X)$ instead of $\Aut_{\Q}(X,\pi_X).$ 

In this article we compute the modular curves $(X_G,\pi_G)$ by computing their function fields and give a method to compute the generators for $G.$ If $X_G$ is $\PP^1_{\Q}$, then the isomorphism class of $(X_G,\pi_G)$ determines $G \subseteq \GL_2(\Zhat)$ up to conjugacy. See \cref{curves-groups} in \cref{sec:Modular Curves} for a proof.

There are infinitely many subgroups $G \subseteq \GL_2(\Zhat)$ for which $X_G$ has genus $0.$ Let us see an example of an infinite family of such groups $G.$ The following example also appears as Remark $1.3$ in \cite{MR3671434}.

\begin{example}\label{serre-curves}

Let $d$ be a squarefree integer. Let $G \subseteq \GL_2(\Zhat)$ be the subgroup consisting of matrices whose image modulo $2$ lies in the unique index $2$ subgroup of $\GL_2(\ZZ/2\ZZ).$ Let $\gamma \colon \GalQ \to \{\pm1\} \simeq \GL_2(\Zhat)/G$ be the quadratic character corresponding to the extension $\Q(\sqrt{d})/\Q.$ There is a unique homomorphism $\phi \colon \Zhat^{\times} \to \GL_2(\Zhat)/G$ such that \[\gamma= \phi \circ \chi_{\cyc}.\] Define \[G_d := \{g \in \GL_2(\Zhat)| gG = \phi(\det g) \};\] it is an open subgroup of $\GL_2(\Zhat)$ of index $2$ that has full determinant and contains $-I.$ The group $G_d \intersect \SL_2(\ZZ)$ does not depend on $d$; i.e., $G_d \intersect \SL_2(\ZZ)=G \intersect \SL_2(\ZZ)$ for every $d.$

The modular curve $X_{G_d}$ is $\PP^1_{\Q}$ with the associated morphism $\pi_{G_d} \colon X_{G_d} \to \PP^1_{\Q}$ described by the rational function $\pi_{G_d}(t)=dt^2+1728.$ The curves $\{(X_{G_d},\pi_{G_d})\}_d$ are $\Qab$-twists of each other. Moreover, $(X_{G_d},\pi_{G_d})$ is a $\Q(\sqrt{d})$-twist of $(X_{G_1},\pi_{G_1}).$

\end{example}

\begin{remark}
Any rational number is of the form $dt^2+1728$ for some squarefree $d$ and $t \in \Q.$ Therefore, for every elliptic curve $E$ over $\Q$ with $j(E) \notin \{0,1728\}$, the image of $\rho_E$ is contained in $G_d$ for some squarefree $d.$ \\
Serre observed that $\rho_E \colon \GalQ \to \GL_2(\Zhat)$ is not surjective for all elliptic curves $E/\Q$, see Proposition $22$ of \cite{MR0387283}.  
\end{remark}

The following theorem states that we can classify all the genus $0$ modular curves that have a rational point via finitely many of them.

\begin{theorem}\label{THM:FIRST}

There are nonconstant $\pi_1,\ldots,\pi_r \in \Q(t)$ and finite abelian subgroups $\calA_i \subseteq \Aut(\PP^1_{\Q},\pi_i) \subseteq \PGL_2(\Q)$ such that the following hold:
 \begin{itemize}
    \item For each $i \in \{1,\ldots,r\}$, $(\PP^1_{\Q},\pi_i)$ is isomorphic to some modular curve $(X_G,\pi_G).$
    
    \item For any modular curve $X_G$ isomorphic to $\PP^1_{\Q}$, there is an $i \in \{1,\ldots,r\}$, such that $(X_G,\pi_G)$ is a $\Qab$-twist of $(\PP^1_{\Q},\pi_i)$ via a cocycle $\gamma \colon \Gal(\Qab/\Q) \to \calA_i \subseteq \Aut(\PP^1_{\Q},\pi_i).$
    
\end{itemize}
Moreover, we can explicitly compute $\pi_i$ and $\calA_i.$
\end{theorem}

Fix a pair $(\pi_i,\calA_i)$ as in \cref{THM:FIRST}. One can show that the twist of $(\PP^1_{\Q},\pi_i)$ by a cocycle $\gamma \colon \Gal(\Qab/\Q) \to \calA_i$ is isomorphic to some modular curve $(X_G,\pi_G).$ The curve $X_G$ has genus $0$ but it need not have a rational point. To give a classification we need to describe the twists with rational points.

We choose a $u_i \in \Q(t)$ such that $\Q(t)^{\calA_i}=\Q(u_i).$ Moreover, $\Aut(\PP^1_{\Q},u_i)=\calA_i.$ We have $\pi_i=J_i \circ u_i$ for a unique rational function $J_i \in \Q(t).$ 

Take any $v \in \Q$ with $J_i(v) \notin \{0,1728,\infty\}.$ Fix a point $P \in u_i^{-1}(v) \subseteq \PP^1(\Qbar).$ Then, there exists a unique cocycle $\gamma \colon \Gal(\Qab/\Q) \to \calA_i$ such that $\sigma(P)=\gamma_\sigma(P)$ for all $\sigma \in \Gal(\Qab/\Q).$ We will construct a rational function $\pi_{\calA_i,v} \in \Q(t)$ such that $(\PP^1_{\Q},J_i \circ \pi_{\calA_i,v})$ is a twist of $(\PP^1_{\Q},\pi_i)$ via $\gamma.$ We will give an explicit $\pi_{\calA_i,v}$ in \cref{THM:SECOND} for the $\calA_i \subseteq \PGL_2(\Q)$, up to conjugacy, that occur in \cref{THM:FIRST}.

\begin{theorem}\label{THM:SECOND}
Let $G$ be an open subgroup of $\GL_2(\Zhat)$ with full determinant containing $-I.$ If $X_G$ is isomorphic to $\PP^1_{\Q}$, then $(X_G,\pi_G)$ is isomorphic to $(\PP^1_{\Q},J_i \circ \pi_{\calA_i,v})$ for some $i \in \{1,\ldots,r\}$ and $v \in \Q$ with $J_i(v) \notin \{0,1728,\infty\}.$ Conversely, any such pair $(\PP^1_{\Q},J_i \circ \pi_{\calA_i,v})$ is isomorphic to a modular curve $(X_G,\pi_G).$ We list a function $\pi_{\calA,v}$ below for various cases of $\calA$ that occur.\begin{itemize}
    \item For $\calA =\{(\begin{smallmatrix}
1 & 0 \\ 
0  & 1
\end{smallmatrix})\}$, define $\pi_{\calA,v}(t)=t.$ 
    
    \item For $\calA =\{(\begin{smallmatrix}
1 & 0 \\ 
0  & 1
\end{smallmatrix}),(\begin{smallmatrix}
-1 & 0 \\ 
0  & 1
\end{smallmatrix})\} \isom \ZZ/2\ZZ$, define $\pi_{\calA,v}(t)=vt^2.$
    
    \item For $\calA =\{(\begin{smallmatrix}
1 & 0 \\ 
0  & 1
\end{smallmatrix}),(\begin{smallmatrix}
0 & \alpha \\ 
1  & 0
\end{smallmatrix})\} \isom \ZZ/2\ZZ$ with $\alpha$ a non-zero rational number which is not a square, define \[\pi_{\calA,v}(t) =\frac{vt^{2}-4\alpha t+\alpha v}{-t^{2}+vt-\alpha}.\] 
    
    \item For $\calA =\{(\begin{smallmatrix}
1 & 0 \\ 
0  & 1
\end{smallmatrix}),(\begin{smallmatrix}
0 & -1 \\ 
1  & -1
\end{smallmatrix}),(\begin{smallmatrix}
1 & -1 \\ 
1  & 0
\end{smallmatrix})\}\isom \ZZ/3\ZZ$, define $\pi_{\calA,v}(t)$ as \[\frac{(-v+3)t^3+(-3v^2+9v-9)t^2+(-3v^3+9v^2-15v)t+(-v^4+3v^3-6v^2-v+3)}{t^3+2vt^2+(v^2+v-3)t+(v^2-3v+1)}.\] 
    
\item For $\calA =\{(\begin{smallmatrix}
1 & 0 \\ 
0  & 1
\end{smallmatrix}),(\begin{smallmatrix}
0 & -1 \\ 
1  & 0
\end{smallmatrix}),(\begin{smallmatrix}
-1 & -1 \\ 
1  & -1
\end{smallmatrix}),(\begin{smallmatrix}
1 & -1 \\ 
1  & 1
\end{smallmatrix})\}\isom \ZZ/4\ZZ$, define $\pi_{\calA,v}(t)$ as \[\frac{-vt^{4} + (8v + 16)t^{3} + (-18v - 96)t^{2} + (8v + 176)t + (7v - 96)}{t^{4} +(v - 8)t^{3} + (-6v + 18)t^{2} + (11v - 8)t + (-6v - 7)}.\] 
    
     \item For $\calA=\{(\begin{smallmatrix}
1 & 0 \\ 
0  & 1
\end{smallmatrix}),(\begin{smallmatrix}
0 & \alpha \\ 
1  & 0
\end{smallmatrix}),(\begin{smallmatrix}
-1 & 0 \\ 
0  & 1
\end{smallmatrix}),(\begin{smallmatrix}
0 & -\alpha \\ 
1  & 0
\end{smallmatrix})\} \isom \ZZ/2\ZZ \times \ZZ/2\ZZ$ with $\alpha$ a non-zero rational number, define $\pi_{\calA,v}(t)$ as \[\frac{(-3v\alpha^2 + v^3)\alpha^2t^4 + (8\alpha^2 - 4v^2)\alpha^2t^3 + 6v\alpha^2t^2 - 8\alpha^2t +
    v}{\alpha^4t^4 - 2v\alpha^2t^3 + (2\alpha^2 + v^2)t^2 - 2vt + 1}.\] 
    
    \item For $\calA=\{(\begin{smallmatrix}
1 & 0 \\ 
0  & 1
\end{smallmatrix}),(\begin{smallmatrix}
0 & -1 \\ 
1  & 0
\end{smallmatrix}),(\begin{smallmatrix}
1 & 1 \\ 
1  & -1
\end{smallmatrix}),(\begin{smallmatrix}
-1 & 1 \\ 
1  & 1
\end{smallmatrix})\} \isom \ZZ/2\ZZ \times \ZZ/2\ZZ$, define $\pi_{\calA,v}(t)$ as \[\frac{(-25v^3 + 160v^2 - 256v)t^4 + (40v^3 - 208v^2 + 256v)t^3+ (-26v^3 + 96v^2 - 64v)t^2+(8v^3 - 16v^2)t-v^3}{(6v^3 - 37v^2 + 64v - 64)t^4+(-11v^3 + 56v^2 - 32v)t^3+ (6v^3 - 30v^2)t^2+(-v^3 +8v^2)t-v^2}.\] 
    
    \item For $\calA=\{(\begin{smallmatrix}
1 & 0 \\ 
0  & 1
\end{smallmatrix}),(\begin{smallmatrix}
0 & -5 \\ 
1  & 0
\end{smallmatrix}),(\begin{smallmatrix}
-1 & 5 \\ 
1  & 1
\end{smallmatrix}),(\begin{smallmatrix}
5 & 5 \\ 
1  & -5
\end{smallmatrix})\} \isom \ZZ/2\ZZ \times \ZZ/2\ZZ$, define $\pi_{\calA,v}(t)$ as $P_1(t)/Q_1(t)$ where 
\begin{equation*}
\begin{aligned}
    P_1(t)=& (-v^5 + 32v^4 - 336v^3 + 1280v^2 - 1600v)t^4 \\
           & + (-4v^6 + 148v^5 - 1744v^4 + 5600v^3 + 16000v^2 -64000v)t^3 \\
           & + (-6v^7 + 252v^6- 3376v^5 + 10880v^4 + 69600v^3 - 256000v^2 - 640000v)t^2 \\
           & +(-4v^8 + 188v^7 - 2864v^6 + 10680v^5 +82400v^4 - 368000v^3 - 1280000v^2)t \\
           & +(-v^9 + 52v^8 - 896v^7 + 4120v^6 + 29500v^5 - 176000v^4 -640000v^3)
\end{aligned}
\end{equation*} and 

\begin{equation*}
\begin{aligned}
    Q_1(t)=& (v^4 - 36v^3 + 240v^2 + 1600v - 14400)t^4 \\
           & +(3v^5 - 128v^4 + 1320v^3 + 1920v^2 - 51200v)t^3 \\
           & + (3v^6 - 149v^5 + 2056v^4 - 3600v^3 -63200v^2 + 64000v)t^2 \\
           & + (v^7- 58v^6 + 1002v^5 - 4400v^4 - 20800v^3 + 96000v^2)t \\
           & +(-v^7 + 26v^6 - 280v^5 + 3500v^4 - 16000v^3 - 160000v^2).
\end{aligned}
\end{equation*}

    \item For $\calA=\{(\begin{smallmatrix}
1 & 0 \\ 
0  & 1
\end{smallmatrix}),(\begin{smallmatrix}
-2 & 2 \\ 
1  & 2
\end{smallmatrix}),(\begin{smallmatrix}
1 & 2 \\ 
1  & -1
\end{smallmatrix}),(\begin{smallmatrix}
0 & -2 \\ 
1  & 0
\end{smallmatrix})\} \isom \ZZ/2\ZZ \times \ZZ/2\ZZ$, define $\pi_{\calA,v}(t)$ as $P_2(t)/Q_2(t)$ where 
\begin{equation*}
\begin{aligned}
    P_2(t)= & vt^4 + (2v^3 + 4v^2 - 24v)t^3 + (3v^5/2 + 6v^4 - 29v^3 - 68v^2 + 184v)t^2 \\
            & + (v^7/2 + 3v^6 - 11v^5 - 62v^4 + 108v^3 + 320v^2 -
    480v)t \\
            & + (v^9/16 + v^8/2 - 5v^7/4 - 14v^6 + 41v^5/4 + 130v^4 -80v^3 - 400v^2 + 400v)
\end{aligned}
\end{equation*} and 

\begin{equation*}
\begin{aligned}
    Q_2(t)=& t^4 + (3v^2/2 + 2v - 8)t^3 + (3v^4/4 +
    7v^3/4 - 5v^2 - 8v + 24)t^2 \\
           & + (v^6/8 + v^5/4 + v^4/4 + 3v^3 -
    10v^2 + 8v - 32)t \\
           &+ (-v^7/16 + 3v^6/8 + 7v^5/2 - 29v^4/4 - 31v^3 + 64v^2 + 16).
\end{aligned}
\end{equation*}

   \end{itemize}
\end{theorem}

\begin{remark}

There are $269$ pairs $(\PP^1_{\Q},\pi)$ and $228$ subgroups $\calA \subsetneq \PGL_2(\Q)$ in \cref{THM:FIRST} corresponding to $109$ genus $0$ congruence subgroups up to conjugacy in $\PGL_2(\ZZ)$. We would like to remark that these numbers are not unique but they are minimal.  
\end{remark}
\subsection{Related Work}

For a complete list of genus $0$ and genus $1$ modular curves $(X_G,\pi_G)$ defined over $\Q$ of prime power level such that $X_G$ has infinitely many rational points, please see \cite{MR3671434}. For a complete classification of all the possible $2$-adic images of Galois representations associated to elliptic curves, please see \cite{MR3500996}. For some cases of composite level modular curves, please refer to the work \cite{MR3957898}.

\subsection{Overview}

In \cref{Overview} we discuss an overview of proofs of \cref{THM:FIRST} and \cref{THM:SECOND}. In \cref{sec:Modular Curves}, we discuss preliminaries and background about modular curves and modular functions. In \cref{sec:Haupt}, we discuss the computation of hauptmoduls for genus 0 congruence subgroups. In \cref{sec:Twists}, we discuss some background material about cocycles and twists and give a method to determine if certain cocycles are coboundaries. In \cref{Horz group decsription}, we give a description of the groups that show up in family of twists. In \cref{sec:families of twists}, we discuss the parametrization of finite Galois extensions $K$ of $\Q$ and give a unique homomorphism $\gamma \colon \Gal(K/\Q) \to \Aut(X_G,\pi_G)$ such that $\gamma$ is a coboundary. In \cref{proof 1}, we give a proof of \cref{THM:FIRST} and in \cref{proof 2} we give a proof of \cref{THM:SECOND}. In \cref{sec:Applications}, we give a procedure to determine generators for an open subgroup $G$ of $\GL_2(\Zhat)$ such that $X_G$ is $\PP^1_{\Q}.$ We also give a method to determine $X_G$ explicitly as a conic from a set of generators for $G.$ Finally, in \cref{sec: Tables}, we give a guide to go through the tables and give our tables in the end.

\subsection{Acknowledgements}

I would like to render my warmest thanks to my advisor David Zywina for suggesting this problem and providing with his valuable comments and guidance throughout my research. All the computations were performed with \texttt{Magma} \cite{MR1484478}. I would also like to thank the anonymous referee for their valuable suggestions.

\subsection{Notations}

We will use $\Qab$ to denote the maximal abelian extension of $\Q.$ We say that a curve is \defi{nice} if it is smooth, projective and geometrically irreducible. We will denote the $j$-invariant of an elliptic curve $E$ by $j(E).$ Let $N$ be a positive integer. We will use $\zeta_N$ to denote $e^{2\pi \dot{\iota}/ N} \in \CC.$ We will denote the $N$-th cyclotomic field $\Q(\zeta_N)$ by $K_N.$ Let $\Q^{cyc}$ be the compositum of cyclotomic fields $K_N$ with $N \ge 1.$ By the Kronecker-Weber theorem $\Qab$ is equal to the field $\Q^{cyc}.$ For a prime $p$, we will use $\Q_p$ to denote the field of $p$-adic numbers. All the topological groups have profinite topology.      

\section{Overview of Proof of \cref{THM:FIRST} and \cref{THM:SECOND}}\label{Overview}

Let us give a brief description of the procedure we use to classify all the modular curves $(X_G,\pi_G)$, up to isomorphism such that $X_G$ is $\PP^1_{\Q}$.\\

For an open group $G \subseteq \GL_{2}(\hat{\ZZ})$ satisfying $-I \in G$ and $\det(G)=\Zhat^{\times}$ define the congruence subgroup $\Gamma := G \intersect \SL_2(\ZZ)$ of $\SL_2(\ZZ).$ As a Riemann surface $X_G(\CC)$ is isomorphic to the compact Riemann surface $X(\Gamma)$ obtained after adding cusps to the quotient $\Gamma \setminus \calH.$ In particular, the genus of $X_{G}$ is equal to the genus of $X(\Gamma)$ which we call the genus of $\Gamma.$ See \cref{sec:Modular Curves} for more details. The connection of these modular curves with congruence subgroups is a key point which we exploit to give a classification.

Fix a genus 0 congruence subgroup $\Gamma \subseteq \SL_2(\ZZ)$ containing $-I$ of level $N$ see \cref{SL2level} for a definition of level of congruence subgroups. Our goal is to compute all the pairs $(X_G,\pi_G)$ such that $X_G$ is $\PP^1_{\Q}$ and $G \intersect \SL_2(\ZZ)=\Gamma.$ 

We compute an explicit function $h$ such that $\CC(X_{\Gamma})=\CC(h)$, see \cref{X_Gamma} for a definition of $X_{\Gamma}.$ We compute the function $\pi_{\Gamma} \in K_N(t)$ which satisfies $\pi_{\Gamma}(h)=j.$ This function describes the morphism $\pi_{\Gamma} \colon X_{\Gamma} \to \PP^{1}_{K_N}.$ 

We then search for a genus 0 modular curve $(X_{G_1},\pi_{G_1})$ satisfying the property that the modular curve $(X_{G_1},\pi_{G_1})$ is isomorphic to $(X_{\Gamma},\pi_{\Gamma})$ over some $K_M$, where $M$ is a multiple of $N$ and is isomorphic to $\PP^1_{\Q}.$ We either succeed in finding such a pair or we observe that there exists no such modular curve $(X_G,\pi_G)$ such that $G \intersect \SL_2(\ZZ)=\Gamma.$

Then, we search for all the genus 0 modular curves $(X_{G},\pi_{G})$ with a rational point which are isomorphic to $(X_{G_1},\pi_{G_1})$ over $\Qab.$ After computing finite number of twists we end up considering homomorphisms $\phi \colon \Gal(\Qab/\Q) \to \calA$, where $\calA$ is a finite subgroup of $\PGL_2(\Q).$ In the \cref{serre-curves}, $\calA$ is $\{(\begin{smallmatrix}
1 & 0 \\ 
0  & 1
\end{smallmatrix}),(\begin{smallmatrix}
-1 & 0 \\ 
0  & 1
\end{smallmatrix})\}.$ Therefore, we end up having an infinite family $\{X_{G_D}\}$ such that two of them become isomorphic over a quadratic extension of $\Q.$

\section{Modular Curves and Modular Functions} \label{sec:Modular Curves}

\subsection{Congruence subgroups}

Let $\Gamma$ be a congruence subgroup of $\SL_{2}(\ZZ).$ \label{SL2level}The level of $\Gamma$ is the smallest positive integer $N$ for which $\Gamma$ contains \[\Gamma(N):=\{A \in \SL_{2}(\ZZ)~|~A \equiv I\pmod{N}\}.\] The group $\SL_2(\ZZ)$ acts on the complex upper half plane $\calH$ by linear fractional transformations. 

We define the extended upper half plane $\calH^{*}$ as $\calH \cup \Q \cup \{\infty\}.$ The action of $\SL_2(\ZZ)$ on $\calH$ naturally extends to the action on $\calH^{*}.$ A \defi{cusp} of $\Gamma$ is a $\Gamma$-orbit of $\Q \cup \{\infty\}.$ After adding cusps to the quotient $\Gamma \setminus \calH$, we get a smooth compact Riemann surface which we denote by $X(\Gamma).$ We will use the notation $X_N$ to denote $X(\Gamma(N))$, where $N$ is a positive integer.

\subsection{Modular functions}

Let $\Gamma$ be a congruence subgroup of $\SL_2(\ZZ).$ A \defi{modular function} $f$ for $\Gamma$ is a meromorphic function of $X(\Gamma)$ i.e., $f$ is a meromorphic function on the upper half plane $\calH$ that satisfies $f(\gamma \tau)=f(\tau)$ for all $\gamma \in \Gamma$ and is meromorphic at the cusps.   

The \defi{width} of the cusp at $\infty$ is the smallest positive integer $w$ such that $(\begin{smallmatrix}
1 & w \\ 
0  & 1
\end{smallmatrix}) \in \Gamma.$ Define $q:=e^{2\pi \dot{\iota} \tau}$ and $q^{1/w}:=e^{2\pi \dot{\iota} \tau/w}$, where $\tau \in \calH.$ If $f$ is a modular function for $\Gamma$, then $f$ has a unique $q$-expansion $f(\tau)=\sum \limits_{n \in \ZZ} c_{n}q^{n/w}$ with $c_{n} \in \CC$ and $c_n=0$ for all but finitely many $n$ less than $0.$

Fix a positive integer $N.$ Let $\calF_N$ denote the field of meromorphic functions of the Riemann surface $X_N$ whose \defi{$q$-expansion} has coefficients in $K_{N}$. For $N=1$, we have $\calF_1=\Q(j)$ where $j$ is the modular $j$-invariant. The first few terms of $q$-expansion of $j$ are $q^{-1}+744+196884q+21493760q^2+\dotsb.$
If $N'$ is a divisor of $N$, then $\calF_{N'} \subseteq \calF_N.$ In particular, we have that $\calF_1 \subseteq \calF_N.$

Take $g \in \GL_2(\ZZ/N\ZZ)$ and let $d=\det(g).$ Then, $g=(\begin{smallmatrix}
1 & 0 \\
0 & d \end{smallmatrix})g'$ for some $g' \in \SL_2(\ZZ/N\ZZ)$. There is a unique right action $*$ of $\GL_2(\ZZ/N\ZZ)$ on $\calF_N$ that satisfies the following properties (see Chapter $2$, section $2$ in \cite{MR648603} for details):

\begin{itemize}
\item If $g \in \SL_2(\ZZ/N\ZZ)$, then \[f*g:=f|_{\gamma},\] where $\gamma \in \SL_{2}(\ZZ)$ is congruent to $g$ modulo $N$ and $f|_{\gamma}(\tau)=f(\gamma \tau)$.

\item If $g=(\begin{smallmatrix}
1 & 0 \\
0 & d \end{smallmatrix})$, then 
\[f*g:=\sigma_{d}(f):=\sum\limits_{n \in \ZZ} \sigma_{d}(c_{n})q^{n/w},\] where $f=\sum\limits_{n \in \ZZ} c_{n}q^{n/w}$ and $\sigma_{d}$ is the automorphism of $K_N$ that satisfies $\sigma_{d}(\zeta_N)=\zeta_N^{d}.$
\end{itemize} 

\begin{proposition}The following properties hold.\label{action on F_N}
\hfill
\begin{romanenum}
    \item $-I$ acts trivially on $\calF_N$ and the action $*$ of $\GL_{2}(\ZZ/N\ZZ)/\{\pm I\}$ on $\mathcal{F}_N$  is faithful.
    
    \item We have $\calF_N^{\GL_2(\ZZ/N\ZZ)/\{\pm I\}}=\Q(j).$
    
    \item \label{opposite}The extension $\mathcal{F}_N$ over $\Q(j)$ is Galois and the action $*$ gives an isomorphism \[\GL_{2}(\ZZ/N \ZZ)/\{\pm I\} \to \Gal(\mathcal{F}_N/\Q(j))^{op},\] where $\Gal(\mathcal{F}_N/\Q(j))^{op}$ is the opposite group of $\Gal(\mathcal{F}_N/\Q(j))$ (i.e., the same underlying set with the group law reversed.)
    
    \item The algebraic closure of $\Q$ in $\calF_{N}$ is $K_{N}$. We have $\mathcal{F}_N^{\SL_{2}(\ZZ/N \ZZ)/\{\pm I\}}=K_N(j).$\label{algclosure}

\end{romanenum}

\end{proposition}

\begin{proof}

Refer to Theorem $6.6$, Chapter $6$ in \cite{MR1291394}.
\end{proof}

\begin{remark}

We require the opposite in \cref{action on F_N}, (\ref{opposite}) since $*$ is a right action.
\end{remark}

\subsection{Modular Curves}

Let $G$ be an open subgroup of $\GL_{2}(\Zhat)$ satisfying $-I \in G$ and $\det(G)=\Zhat^{\times}.$\label{level G} We define the \defi{level} of $G$ as the smallest positive integer $M$ such that $G$ is the inverse image of its image under the reduction modulo $M$ map $\pi_{M} \colon \GL_2(\Zhat) \to GL_2(\ZZ/M\ZZ).$ 

Let $M$ be the level of $G$. Let $G_M \subseteq \GL_2(\ZZ/M\ZZ)$ be the image of $G$ under $\pi_M.$ Since $G_M$ has full determinant, it will follow that $\Q$ is algebraically closed in $\calF_M^{G_M}$ by \cref{action on F_N}, (\ref{opposite}). The field $\calF_M^{G_M}$ has transcendence degree $1$ over $\Q$ since it is a finite extension of $\Q(j)$. 

\begin{definition}

The \defi{modular curve} $X_G$ is the nice curve over $\Q$ whose function field is $\calF_M^{G_M}.$ The map $\pi_G \colon X_G \to \PP^1_{\Q}$ is the nonconstant morphism corresponding to the inclusion  $\Q(j) \subseteq \calF_M^{G_M}.$
\end{definition}

The following lemma partially explains how $(X_G,\pi_G)$ parametrises elliptic curves. 

\begin{lemma}
For an elliptic curve $E$ defined over $\Q$ with $j(E) \notin \{0,1728\}$, the group $\rho_E(\GalQ)$ is conjugate in $\GL_2(\Zhat)$ to a subgroup of $G^{t}$ if and only if $j(E) \in \pi_G(X_G(\Q)).$
\end{lemma} 

\begin{proof}

Refer to section 3.5 of \cite{1508.07660}. The transpose shows up because $*$ is a right action.
\end{proof}

\begin{remark}
One could also define $X_{G}$ as the generic fiber of the coarse space for the stack $M_{G}$ which parametrizes elliptic curves with $G$-level structure (refer to \cite{MR0337993} chapter $4$, section 3 for details).

\end{remark}

In the following lemma we prove that classifying modular curves $(X_G,\pi_G)$ up to isomorphism is equivalent to classifying open subgroups $G$ of $\GL_2(\Zhat)$ up to conjugation in $\GL_2(\Zhat).$

\begin{lemma} \label{curves-groups}
Let $G$ and $G'$ be two open subgroups of $\GL_2(\Zhat)$ with full determinant and containing $-I$. Then, $G$ and $G'$ are conjugate in $\GL_2(\Zhat)$ if and only if $(X_G,\pi_G)$ and $(X_{G'},\pi_{G'})$ are isomorphic. 
\end{lemma}

\begin{proof}

Let us assume that $G$ and $G'$ are conjugate in $\GL_2(\Zhat)$,i.e., $G'=gGg^{-1}$ for some $g \in \GL_2(\Zhat).$ It is easy to check that $G$ and $G'$ have the same level. Let $M$ be that level. Then, the map $\calF_M^{G_M} \to \calF_M^{G'_M}$ that sends $f$ to $f*g^{-1}$ is an isomorphism that fixes $\Q(j).$ Therefore, we get a corresponding isomorphism between $(X_G,\pi_G)$ and $(X_{G'},\pi_{G'})$.

Let us assume that $(X_G,\pi_G)$ and $(X_{G'},\pi_{G'})$ are isomorphic, i.e., there exists an isomorphism $\bar{\phi} \colon \Q(X_G) \to \Q(X_{G'})$ that fixes $\Q(j).$ Let $M$ and $M'$ be respectively the levels of $G$ and $G'.$ We will use $G$ and $G'$ to denote their image modulo $M$ and $M'$ respectively. We have the inclusions $\Q(j) \subseteq \Q(X_G) \subseteq \calF_{MM'}$ and $\Q(j) \subseteq \Q(X_{G'}) \subseteq \calF_{MM'}.$ There exists an element $\phi \in \Gal(\calF_{MM'}/\Q(j))$ such that $\phi|_{\Q(X_{G'})}=\bar{\phi}.$ Let $g \in \GL_2(\ZZ/MM'\ZZ)$ be the corresponding element from \cref{action on F_N}, (\ref{opposite}), i.e., the map $\phi \colon \Q(X_G) \to \Q(X_{G'})$ sends $f$ to $f*g.$ We will now show that $gG'g^{-1} \subseteq G.$ Let $a \in gG'g^{-1}$, i.e., $a=ga'g^{-1}$ for some $a' \in G'.$ For any $f \in \Q(X_G)$, $(f*g)*a'=f*g$. Therefore, $f*(ga'g^{-1})=f.$ Hence, $a=ga'g^{-1} \in G.$ Since, the degree of $\pi_G$ is equal to the degree of $\pi_{G'}$ we can conclude that $gG'g^{-1} = G.$ 
\end{proof}

Let $\Gamma$ be a congruence subgroup of level $N$ that contains $-I.$ Let $\Gamma_N \subseteq \SL_2(\ZZ/N\ZZ)$ be the image of $\Gamma$ under reduction modulo $N.$

\begin{definition}\label{X_Gamma}

The algebraic curve $X_{\Gamma}$ is the nice curve over $K_N$ whose function field is $\calF_N^{\Gamma_N}.$ The map $\pi_\Gamma \colon X_{\Gamma} \to \PP^1_{K_N}$ is the morphism corresponding to the inclusion $K_N(j) \subseteq \calF_N^{\Gamma_N}.$
\end{definition}

Using the inclusion $K_N \subsetneq \CC$, we can identify $X_{\Gamma}(\CC)$ with the Riemann surface $X(\Gamma).$

Using \cref{action on F_N}, (\ref{algclosure}) we know that the field $K_M(X_G)$ which is the function field of the base extension of $X_G$ to $K_M$ is isomorphic to the function field $K_M(X_{\Gamma})$, where $\Gamma$ is conjugate to the congruence subgroup $G \intersect \SL_2(\ZZ).$ Equivalently, there exists an isomorphism $f \colon (X_{\Gamma})_{K_M} \to (X_G)_{K_M}$ that satisfies $\pi_G \circ f = \pi_{\Gamma}.$

\section{Computing hauptmoduls}\label{sec:Haupt}

Let $\Gamma$ be any genus $0$ congruence subgroup of $\SL_{2}(\ZZ)$ of level $N$. Since $\Gamma$ has genus 0, the function field $\CC(X(\Gamma))$ over $\CC$ is generated by a single modular function $h$. We can choose $h$ so that it has a simple pole at any fixed point of $X(\Gamma).$  

A \defi{hauptmodul} of $\Gamma$ is a meromorphic function $h$ on $X(\Gamma)$ with a unique pole at the cusp $\infty$ and this pole is simple. We have $\CC(X(\Gamma))=\CC(h).$ 

Fix a hauptmodul $h$ of $\Gamma.$ Every other hauptmodul of $\Gamma$ is of the form $ah+b$ for a unique $a \in \CC^{\times}$ and $b \in \CC.$ We say that a hauptmodul is \defi{normalized} if the leading coefficient of its $q$-expansion is $1$ and the constant coefficient is $0.$ A normalized hauptmodul of $\Gamma$ is unique and all the coefficients of its $q$-expansion lie in $K_N.$ This can be proven directly; it will also follow from our computations. To get a normalized hauptmodul one can rescale a given hauptmodul such that the constant coefficient is $0$ and the leading coefficient is $1.$ 

In this section, we will discuss computation of hauptmoduls. Our goal is to describe how to express the normalized hauptmodul of $\Gamma$ in terms of explicit Siegel functions. This explicit description allows us to compute arbitrarily many terms of the $q$-expansion of the hauptmodul at each cusp of $\Gamma.$
\\
Suppose that $h$ is the normalized hauptmodul of $\Gamma.$ Let $\gamma \Gamma \gamma^{-1}$ be a conjugate of $\Gamma$, where $\gamma \in \SL_2(\ZZ)$. Then, $h':=h|_{\gamma^{-1}}$ is a meromorphic function on $X(\gamma \Gamma \gamma^{-1}).$ If the pole of $h'$ is at $\infty$, then $h'$ is a hauptmodul. Otherwise, $1/(h'-a_0)$ is a meromorphic function on $X(\gamma \Gamma \gamma^{-1})$ with the pole at $\infty$, where $a_0$ is the constant coefficient of $h'.$ Therefore, it suffices to compute the hauptmodul for representatives of the conjugacy classes of genus $0$ congruence subgroups in $\SL_2(\ZZ).$  

There are only finitely many congruence subgroups of genus $0$. Moreover, the genus $0$ congruence subgroups of $\SL_2(\ZZ)$ containing $-I$ are listed, up to conjugation in $\GL_2(\ZZ)$, by Cummins and Pauli in \cite{MR2016709}. To obtain a list up to conjugation of genus $0$ congruence subgroups containing $-I$ in $\SL_2(\ZZ)$ we can do the following. Let $A=(\begin{smallmatrix}
-1 & 0 \\ 
0  & 1
\end{smallmatrix})$ for the rest of this section. Fix a genus 0 congruence subgroup $\Gamma$ in \cite{MR2016709}. Let $\Gamma^{'}:=A\Gamma A^{-1}.$ It can be checked that $\Gamma$ and $\Gamma^{'}$ represent at most two conjugacy classes in $\SL_2(\ZZ).$

We will use the labels from \cite{MR2016709} to denote their explicit genus 0 congruence subgroups whenever we compute the hauptmodul with respect to the set of generators listed in \cite{MR2016709}.

\subsection{Siegel functions}
We will briefly discuss the definition and key properties of Siegel functions.

Take $a =(a_{1},a_{2}) \in \Q^{2} \setminus \ZZ^{2}.$ Define $q_{z}:=e^{2\pi \dot{\iota} z}$ and recall that $q=e^{2\pi \dot{\iota} \tau}$, where $\tau \in \calH$ and $z=a_{1}\tau+a_{2}$.  The \defi{Siegel function} $g_{a}$ is the holomorphic function $\calH \to \CC^{\times}$ given by \[g_{a}(\tau)=-q^{(1/2)B_{2}(a_{1})}e^{2\pi \dot{\iota} a_{2}(a_{1}-1)/2}(1 - q_{z}){\displaystyle \prod_{n=1}^{\infty}(1 - q^{n}q_{z})(1 - q^{n}/q_{z})}\] where  $B_{2}(X) = X^2 - X + 1/6$ is the second Bernoulli polynomial.

\begin{lemma}\label{Siegel functions}
 Take $a =(a_{1},a_{2}) \in \Q^{2} \setminus \ZZ^{2}$ and $b =(b_{1},b_{2}) \in \ZZ^{2}.$ The following properties are satisfied by Siegel functions:
\begin{romanenum}
    \item $g_{-a}=-g_a.$
    \item $g_{a+b}(\tau)=\epsilon(a,b)g_{a}(\tau)$, where $\epsilon(a,b)=(-1)^{b_{1}+b_{2}+b_{1}b_{2}}e^{2\pi \dot{\iota} (b_{2}a_{1}-a_{2}b_{1})/2}$ 
    \item For any $\gamma \in \SL_2(\ZZ)$, we have $g_{a}(\gamma \tau)=\varepsilon(\gamma)g_{a \gamma}(\tau),$ where $\varepsilon \colon \SL_2(\ZZ) \to \CC^{\times}$ is the group homomorphism satisfying $\varepsilon((\begin{smallmatrix}
1 & 1 \\
0 & 1 \end{smallmatrix}))=\zeta_{12}$ and $\varepsilon((\begin{smallmatrix}
0 & 1 \\
-1 & 0 \end{smallmatrix}))=\zeta_4.$ 
\end{romanenum}

\end{lemma}

\begin{proof}

Property $(ii)$ follows from K2 in Section 1, Chapter 2 of \cite{MR648603}. Property $(iii)$ follows using K0 and K1 of Section 1, Chapter 2 of \cite{MR648603} and the following transformation property of Dedekind Eta function. Recall that the Dedekind Eta function is a holomorphic function on $\calH$ given by $\eta(\tau)=q^{1/24}\prod_{n=1}^{\infty}(1-q^n)$ and for $\gamma=(\begin{smallmatrix}
a & b \\
c & d \end{smallmatrix}) \in \SL_2(\ZZ)$, $\eta(\gamma \tau)^2=\varepsilon(\gamma)(c\tau+d)\eta(\tau)^2.$ Property $(i)$ follows from property $(iii)$ by taking $\gamma=-I.$

\end{proof}
Let $A_{N}$ be the set of elements in $(1/N)\ZZ^{2} \setminus \ZZ^{2}$ satisfying one of the following conditions.

\begin{itemize}
    \item  $0 < a_{1} < 1/2$ and $0 \le a_{2} < 1$
    \item $a_{1}=0$ and $0 < a_{2} \le 1/2$
    \item $a_{1}=1/2$ and $0 \le a_{2} \le 1/2$
\end{itemize}

For each $a \in N^{-1}\ZZ^2 \setminus \ZZ^2$, there is a unique element $a'\in A_N$ such that $a=\pm a'+ m$ for some $m$ in $\ZZ^2.$ Using properties $(i)$ and $(ii)$ of \cref{Siegel functions} we get that $g_a= \zeta g_{a'},$ where $\zeta$ is an explicit $2N$-th root of unity.

The map $(N^{-1}\ZZ^2 \setminus \ZZ^2) \times \SL_2(\ZZ) \to N^{-1}\ZZ^2 \setminus \ZZ^2$ that sends $(a,\gamma) \mapsto a\gamma$ is a right action of $\SL_2(\ZZ)$ on $N^{-1}\ZZ^2 \setminus \ZZ^2$. This induces a right action of $\SL_2(\ZZ)$ on $A_N$ given by $a*\gamma:= (a\gamma)'.$ 

The following is Lemma $4.2$ from section $4$ of \cite{MR3671434} which gives us an explicit modular function for $\Gamma$ which we use later to construct hauptmoduls.

\begin{lemma}

Let $O$ be a $\Gamma$-orbit of $A_{N}$ and let $g_{O}$ be the holomorphic function $\calH \to \CC^{\times}$ which is defined as $\prod_{a \in O} g_{a}$. Then, $g_{O}^{12N}$ is a modular function for $\Gamma$. Every pole and zero of $g_{O}^{12N}$ on $X(\Gamma)$ is a cusp. 
\end{lemma}

\subsection{Multiple cusps}

Assume that $\Gamma$ has more than one cusp. In this part, we will discuss how to compute a hauptmodul for $\Gamma.$

Let $O_{1},\dots,O_{n}$ be the distinct $\Gamma$-orbits of $A_{N}$. For each $O_{i}$, let $D_{i}$ be the divisor of $g_{O_{i}}^{12N}$ on $X(\Gamma)$. The divisor $D_i$ is supported on the cusps of $\Gamma$. Let $P_1,\dots,P_r$ be the distinct cusps of $\Gamma$ and let $w_j$ be the width of the cusp $P_j$. For each $1 \le j \le r$, let $A_j$ be a matrix in $\SL_2(\ZZ)$ such that $A_j.\infty=s_{j}$, where $s_j \in \Q \cup \{\infty\}$ is a representative of the cusp $P_{j}$. Lemma $4.3$ of \cite{MR3671434} states that \[D_{i}=
\sum\limits_{j=1}^{r}\bigg(6Nw_{j}\sum\limits_{a \in O_i}B_{2}(\{(aA_{j})_{1}\})\bigg).P_{j},\] where for $x \in \RR$, $\{x\}$ satisfies $0 \le \{x\} < 1$ and $x - \{x\} \in \ZZ.$

The following is Lemma $4.4$ from section $4$ of \cite{MR3671434}; we use it to explicitly construct hauptmoduls when $\Gamma$ has multiple cusps.

\begin{lemma}
Suppose there is an $m \in \ZZ^{n}$ such that $\sum\limits_{i=1}^{n} m_{i}D_{i} = -12N.P_{1}+12N.P_{2}$ where $P_{1}$ is the cusp at infinity and $P_2$ is any other cusp of $\Gamma$. Let $0 \le e < 2N^{2}$ be an integer satisfying \[e \equiv \sum\limits_{i=1 }^{n}m_{i}\sum\limits_{a\in O_{i}} Na_{2}(N -Na_{1}) \pmod{2N^{2}}.\] Then $h := (\zeta_{2N^2})^{e}\prod_{ i=1 }^{n} g_{ O_{i}}^{m_{i}}$ is a hauptmodul for $\Gamma$ whose $q$-expansion has coefficients in $K_{N}$. On $X(\Gamma)$, we have $\divv(h) = -P_{1}+P_{2}$.
\end{lemma}

To verify the existence of an $n$-tuple $m$ as described in the previous lemma we used the generators for $\Gamma$ as described in \cite{MR2016709} and the set of generators obtained after conjugation with $A$ (if $\Gamma$ and $\Gamma^{'}$ are not conjugate in $\SL_2(\ZZ)$). The existence of $n$-tuple was verified using \texttt{Magma} for all genus 0 congruence subgroups with multiple cusps except those with labels $10E^0$, $12F^0$ and $12H^0$. It can be checked that the congruence subgroups with labels $10E^0$, $12F^0$ and $12H^0$ are listed up to conjugation in $\SL_2(\ZZ).$ We describe hauptmoduls for $10E^0$ and $12F^0$ later in the section. For the genus 0 congruence subgroup with label $12H^0$, the $n$-tuple exists when we first conjugate the generators with $(\begin{smallmatrix}
1 & 0 \\ 
4  & 1
\end{smallmatrix})$.

We describe the hauptmoduls for congruence subgroups with label $10E^0$ and $12F^0$ below. These hauptmoduls are computed with respect to a set of generators modulo $N$ denoted by $S_{\Gamma}.$ Here $N$ is the level of $\Gamma.$

\begin{remark}

In the following cases the function $h$ comes from section $5.3$ of \cite{Rademacher}. Unfortunately, their tables are no longer available. They compute modular functions in terms of "generalized Dedekind eta functions". We take their function $h$ and then essentially adjust the constants to get a modular function for $\Gamma.$
\end{remark}

\begin{romanenum}

\item Consider the group $\Gamma$ with label $12F^0.$ Here $g_{(a_{1},a_{2})}$ denotes the Siegel function $g_{(a_{1}/12,a_{2}/12)}(\tau).$ 

The function $h$ listed below is fixed by $\Gamma$. After normalizing the $q$-expansion of $h$ we find that the coefficients lie in $K_{12}$ and it has a unique simple pole at the cusp $\infty.$    
    \begin{equation*}
        \begin{aligned}
        h&=g_{(3,11)}g_{(3,8)}g_{(6,5)}g_{(6,1)}g_{(3,7)}g_{(3,4)}g_{(1,8)}g_{(1,5)}g_{(5,4)}g_{(5,1)}g_{(2,1)}g_{(2,7)}\\
    &+\zeta_{24}^3g_{(5,2)}g_{(1,3)}g_{(1,6)}g_{(5,11)}g_{(4,3)}g_{(4,7)}g_{(4,9)}g_{(4,1)}g_{(5,3)}g_{(5,6)}g_{(1,7)}g_{(1,10)}.
     \end{aligned}
    \end{equation*}
    
    \item Consider the group $\Gamma$ with label $10E^0.$ Here $g_{(a_{1},a_{2})}$ denotes the Siegel function $g_{(a_{1}/10,a_{2}/10)}(\tau).$ 

The function $h$ listed below is fixed by $\Gamma$. After normalizing the $q$-expansion of $h$ we find that the coefficients lie in $K_{10}$ and it has a unique simple pole at the cusp $\infty.$
    \begin{equation*}
        \begin{aligned}
        h&=g_{(1,4)}g_{(5,2)}g_{(5,4)}g_{(3,2)}(g_{(1,3)}g_{(3,9)}g_{(1,5)}g_{(3,5)})^{-1}\\
    &+\zeta_{10}^3 g_{(1,7)}g_{(1,1)}g_{(3,1)}g_{(3,3)}(g_{(0,1)}g_{(2,3)}g_{(4,1)}g_{(0,3)})^{-1}\\
    &+\zeta_{10}^2 g_{(2,5)}g_{(4,5)}g_{(2,1)}g_{(4,7)}(g_{(1,2)}g_{(1,6)}g_{(3,6)}g_{(3,8)})^{-1}.
        \end{aligned}
    \end{equation*}

\end{romanenum}

\subsection{Single cusp} In this part, we will assume that $\Gamma$ has a single cusp. In this case, we cannot express hauptmodul of $\Gamma$ as a product of Siegel functions (since such functions are modular units, i.e., their zeroes and poles lie at the cusps).\\
Let us first assume that there exists a proper genus 0 congruence subgroup $\Gamma^{'} \subseteq \Gamma$ containing $-I$ of level $N$ such that $[\Gamma:\Gamma^{'}]$ is equal to the number of cusps of $\Gamma^{'}$. From the previous subsection, we can construct a hauptmodul $h^{'}$ of $\Gamma^{'}.$ 

\begin{lemma}

The modular function $h:=\sum\limits_{\alpha \in \Gamma^{'}\backslash \Gamma} h^{'}|_{\alpha}$ is a hauptmodul for $\Gamma$ with coefficients in $K_{N}.$ 
\end{lemma}

\begin{proof}
 Refer to \cite{MR3671434}, section 4.3.2.
\end{proof}

\subsection{Remaining cases}In this part, we express the hauptmodul for remaining single cusp cases in terms of Siegel functions. In each of these cases there exists a proper congruence subgroup $\Gamma^{'} \subseteq \Gamma$ containing $-I$ of level $N$ (not necessarily of genus $0$) such that $[\Gamma:\Gamma^{'}]$ is equal to the number of cusps of $\Gamma^{'}$. We take a modular function $f$ for $\Gamma^{'}$. The function $f$ comes from section $5.3$ of \cite{Rademacher}. Then, we compute the function $h:=\sum\limits_{\alpha \in \Gamma^{'}\backslash \Gamma} f|_{\alpha}$. It is a modular function for $\Gamma$ whose unique simple pole lies at infinity. This follows from section 4.3.2 of \cite{MR3671434}. After normalizing the $q$-expansion of $h$, we obtain a hauptmodul for $\Gamma.$ We give $f$, a set $S$ of representatives of $\Gamma^{'}\backslash \Gamma$ and the hauptmodul $h.$    

\begin{romanenum}
    \item Consider the group $\Gamma$ with label $14A^0$. Here $g_{(a_{1},a_{2})}$ denotes the Siegel function $g_{(a_{1}/14,a_{2}/14)}(\tau).$ The set $S=\{(\begin{smallmatrix}
1 & 0 \\ 
0  & 1
\end{smallmatrix}),(\begin{smallmatrix}
0 & 5 \\ 
11  & 7
\end{smallmatrix}),(\begin{smallmatrix}
0 & 1 \\ 
13  & 13
\end{smallmatrix}),(\begin{smallmatrix}
0 & 3 \\ 
9  & 13
\end{smallmatrix})\}$,
\\$f:=(\zeta_{14}^5 - \zeta_{14}^4 + \zeta_{14}^3 - \zeta_{14}^2 + \zeta_{14} - 1)g_{(2,0)}g_{(2,10)}g_{(4,2)}
    g_{(4,4)}g_{(6,4)}g_{(6,12)}.$ The resulting hauptmodul is
    \begin{equation*}
    \begin{aligned}
    h&=g_{(0,2)}g_{(2,8)}g_{(2,12)}g_{(4,0)}g_{(4,6)}g_{(4,12)}\\
    &-\zeta_{14}^4g_{(0,4)}g_{(2,2)}g_{(2,4)}g_{(2,6)}g_{(6,2)}
    g_{(6,8)}\\
    &+\zeta_{14} g_{(0,6)}g_{(4,8)}g_{(4,10)}g_{(6,0)}g_{(6,6)}g_{(6,10)}\\
    &+(\zeta_{14}^5 - \zeta_{14}^4 + \zeta_{14}^3 - \zeta_{14}^2 + \zeta_{14} - 1)g_{(2,0)}g_{(2,10)}g_{(4,2)}
    g_{(4,4)}g_{(6,4)}g_{(6,12)}.
    \end{aligned}
    \end{equation*}
    
    \item Consider the group $\Gamma$ when it is the conjugate of the congruence subgroup with label $14A^0$ by A. Here $g_{(a_{1},a_{2})}$ denotes the Siegel function $g_{(a_{1}/14,a_{2}/14)}(\tau).$
    The set $S=\{(\begin{smallmatrix}
1 & 0 \\ 
0  & 1
\end{smallmatrix}),(\begin{smallmatrix}
0 & -5 \\ 
-11  & 7
\end{smallmatrix}),(\begin{smallmatrix}
0 & -1 \\ 
-13  & 13
\end{smallmatrix}),(\begin{smallmatrix}
0 & -3 \\ 
-9  & 13
\end{smallmatrix})\},$ $f:=\zeta_{14}^3g_{(2,0)}g_{(2,4)}g_{(4,12)}
    g_{(4,10)}g_{(6,10)}g_{(6,2)}.$ The resulting hauptmodul is
    \begin{equation*}
        \begin{aligned}
        h&=-g_{(0,2)}g_{(2,6)}g_{(2,2)}g_{(4,0)}g_{(4,8)}g_{(4,2)}\\
    &+g_{(0,4)}g_{(2,12)}g_{(2,10)}g_{(2,8)}g_{(6,12)}
    g_{(6,6)}\\
    &+\zeta_{14}^2 g_{(0,6)}g_{(4,6)}g_{(4,4)}g_{(6,0)}g_{(6,8)}g_{(6,4)}\\
     &- \zeta_{14}^3g_{(2,0)}g_{(2,4)}g_{(4,12)}
    g_{(4,10)}g_{(6,10)}g_{(6,2)}.
    \end{aligned}
    \end{equation*}
    
    \item Consider the group $\Gamma$ with label $15A^0$. Here $g_{(a_{1},a_{2})}$ denotes the Siegel function $g_{(a_{1}/15,a_{2}/15)}(\tau).$ The set $S=\{(\begin{smallmatrix}
1 & 0 \\ 
0  & 1
\end{smallmatrix}),(\begin{smallmatrix}
0 & 1 \\ 
14  & 14
\end{smallmatrix}),(\begin{smallmatrix}
1 & 1 \\ 
14  & 0
\end{smallmatrix})\},$ $f:=g_{(0,3)}g_{(0,6)}g_{(3,12)}g_{(6,9)}.$ The resulting hauptmodul is
\begin{equation*}
    \begin{aligned}
    h&=g_{(0,3)}g_{(0,6)}g_{(3,12)}g_{(6,9)}\\
    &-\zeta_{15}^{2}g_{(3,6)}g_{(3,3)}g_{(6,12)}g_{(6,6)}\\
    &+(\zeta_{15}^5+1)g_{(3,9)}g_{(3,0)}g_{(6,0)}g_{(6,3)}.
    \end{aligned}
\end{equation*}

    \item Consider the group $\Gamma$ when it is the conjugate of the congruence subgroup with label $15A^0$ by A. Here $g_{(a_{1},a_{2})}$ denotes the Siegel function $g_{(a_{1}/15,a_{2}/15)}(\tau).$
    The set $S=\{(\begin{smallmatrix}
1 & 0 \\ 
0  & 1
\end{smallmatrix}),(\begin{smallmatrix}
0 & -1 \\ 
-14  & 14
\end{smallmatrix}),(\begin{smallmatrix}
1 & -1 \\ 
-14  & 0
\end{smallmatrix})\},$  $f:=g_{(0,3)}g_{(0,6)}g_{(3,3)}g_{(6,6)}.$ The resulting hauptmodul is
    \begin{equation*}
        \begin{aligned}
        h&=g_{(0,3)}g_{(0,6)}g_{(3,3)}g_{(6,6)}\\
    &+\zeta_{15}g_{(3,9)}g_{(3,12)}g_{(6,9)}g_{(6,3)}\\
    &-\zeta_{15}^5g_{(3,6)}g_{(3,0)}g_{(6,0)}g_{(6,12)}.
        \end{aligned}
    \end{equation*}

    \item Consider the group $\Gamma$ with label $21A^0$. Here $g_{(a_{1},a_{2})}$ denotes the Siegel function $g_{(a_{1}/21,a_{2}/21)}(\tau).$ The set $S=\{(\begin{smallmatrix}
1 & 0 \\ 
0  & 1
\end{smallmatrix}),(\begin{smallmatrix}
0 & 1 \\ 
20  & 6
\end{smallmatrix}),(\begin{smallmatrix}
0 & 2 \\ 
10  & 15
\end{smallmatrix}),$ $f:=g_{(0,6)}g_{(3,18)}g_{(3,3)}g_{(3,9)}g_{(3,12)}g_{(6,9)}g_{(6,12)}g_{(9,0)}.$ The resulting hauptmodul is
\begin{equation*}
    \begin{aligned}
    h&=-g_{(0,3)}g_{(3,15)}g_{(6,15)}g_{(6,18)}g_{(6,0)}g_{(6,3)}g_{(9,18)}g_{(9,9)}\\
    &+g_{(0,6)}g_{(3,18)}g_{(3,3)}g_{(3,9)}g_{(3,12)}g_{(6,9)}g_{(6,12)}g_{(9,0)}\\
    &-\zeta_{21}^{6}g_{(0,9)}g_{(3,0)}g_{(3,6)}g_{(6,6)}g_{(9,12)}g_{(9,15)}g_{(9,3)}g_{(9,6)}.
    \end{aligned}
\end{equation*}

    \item Consider the group $\Gamma$ when it is the conjugate of the congruence subgroup with label $21A^0$ by A. Here $g_{(a_{1},a_{2})}$ denotes the Siegel function $g_{(a_{1}/21,a_{2}/21)}(\tau).$ The set $S=\{(\begin{smallmatrix}
1 & 0 \\ 
0  & 1
\end{smallmatrix}),(\begin{smallmatrix}
0 & -1 \\ 
-20  & 6
\end{smallmatrix}),(\begin{smallmatrix}
0 & -2 \\ 
-10  & 15
\end{smallmatrix}),$ $f:=(-\zeta_{21}^{11}-\zeta_{21}^4)g_{(0,6)}g_{(3,3)}g_{(3,18)}g_{(3,12)}g_{(3,9)}g_{(6,12)}g_{(6,9)}g_{(9,0)}.$ The resulting hauptmodul is
    \begin{equation*}
        \begin{aligned}
        h&=g_{(0,3)}g_{(3,6)}g_{(6,6)}g_{(6,3)}g_{(6,0)}g_{(6,18)}g_{(9,3)}g_{(9,12)}\\
    &+(-\zeta_{21}^{11}-\zeta_{21}^4)g_{(0,6)}g_{(3,3)}g_{(3,18)}g_{(3,12)}g_{(3,9)}g_{(6,12)}g_{(6,9)}g_{(9,0)}\\
    &+(\zeta_{21}^{11} - \zeta_{21}^9 + \zeta_{21}^8 - \zeta_{21}^6 + \zeta_{21}^4 - \zeta_{21}^3 + \zeta_{21} - 1)g_{(0,9)}g_{(3,0)}g_{(3,15)}g_{(6,15)}g_{(9,9)}g_{(9,6)}g_{(9,18)}g_{(9,15)}.
        \end{aligned}
    \end{equation*}

    \item When $\Gamma$ has level $11.$ A hauptmodul in terms of Siegel functions for any genus 0 congruence subgroup of level $11$ is described in section 4.3.3 of \cite{MR3671434}.

\end{romanenum}

\section{Twists}\label{sec:Twists}

\subsection{Background and preliminaries}

For background in Galois cohomology see section 5.1, Chapter 1 of \cite{MR1466966}. We recall the definitions and concepts that we need for reading this article.  

Let $G$ be a topological group and $A$ be a $G$-group i.e., $A$ is a topological group with the discrete topology, with a continuous $G$-action that respects the group law of $A$.

\begin{definition}

A \defi{cocycle} $\zeta \colon G \to A$, is a continuous function satisfying the \defi{cocycle property}, i.e., $\zeta(\sigma \tau)=\zeta(\sigma) \cdot \sigma(\zeta(\tau))$ for every $\sigma,\tau \in G.$ We say that two cocycles $\zeta_1$ and $\zeta_2$ are \defi{cohomologous} if there is an $a \in A$ such that $\zeta_1(\sigma)=a \cdot \zeta_2(\sigma) \cdot \sigma(a^{-1})$ for every $\sigma \in G.$  
\end{definition}

It is easy to check that being cohomologous is an equivalence relation. We will denote the set of cocycles modulo this relation by $H^{1}(G,A).$ Let $e \colon G \to A$ be the cocycle which sends $g$ to $1$ for every $g \in G$, where $1$ is the identity element of $A.$ We say that a cocycle $\zeta \colon G \to A$ is a \defi{coboundary} if $\zeta$ is cohomologous to $e.$

Let $K$ and $K'$ be Galois extensions of $\Q$ such that $K \intersect K'=\Q.$ Let $KK'$ denote the compositum of $K$ and $K'.$ Let $i \colon \Gal(K/\Q) \times \Gal(K'/\Q) \to \Gal(KK'/\Q)$ be the isomorphism whose inverse is $g \mapsto (g|_K,g|_{K'}).$ Let $A$ be a $\Gal(KK'/\Q)$-group for which $\Gal(KK'/K)$ acts trivially. There is a natural action of $\Gal(K/\Q)$ on $A.$ We let $\Gal(K'/\Q)$ act trivially on $A.$

Let $\phi \colon \Gal(KK'/\Q) \to A$ be a cocycle. We have assumed that the group $\Gal(KK'/\Q)$ acts via $\Gal(K/\Q)$, i.e., $i((\sigma,\tau))(a)=\sigma(a)$ for every $(\sigma,\tau) \in \Gal(K/\Q) \times \Gal(K'/\Q)$ and $a \in A.$

\begin{lemma} \label{breakup of cocs}
Let $\zeta \colon \Gal(K/\Q) \to A$ be defined as $\zeta(\sigma)=\phi(i(\sigma,1))$. Let $\gamma \colon \Gal(K'/\Q) \to A$ be defined as $\gamma(\tau)=\phi(i(1,\tau)).$  
The following properties hold:
\begin{romanenum}
    \item The maps $\zeta$ and $\gamma$ are cocycles. In particular, $\gamma$ is a homomorphism. Further, the image of $\gamma$ is contained in the set \[\{a \in A|~ \sigma(a)=\zeta(\sigma)^{-1}\cdot a\cdot \zeta(\sigma)~\forall~\sigma \in \Gal(K/\Q)\}.\] 

\item Conversely, if $\zeta \colon \Gal(K/\Q) \to A$ and $\gamma \colon \Gal(K'/\Q) \to A$ are two cocycles such that the image of $\gamma$ is contained in the set 
\[\{a \in A|~ \sigma(a)=\zeta(\sigma)^{-1}\cdot a\cdot \zeta(\sigma)~\forall~\sigma \in \Gal(K/\Q)\},\] then the map $\phi \colon \Gal(K/\Q) \times \Gal(K'/\Q) \to A$ defined as $\phi((\sigma,\tau))=\zeta(\sigma)\cdot \sigma(\gamma(\tau))$ is a cocycle.

\end{romanenum}
\end{lemma}

\begin{proof}

\item \begin{romanenum}
    \item We will show that $\zeta$ satisfies the cocycle property. Using cocycle property of $\phi$ we have that for any $\sigma_1$, $\sigma_2 \in \Gal(K/\Q)$, $\zeta(\sigma_1 \sigma_2)=\phi(i(\sigma_1\sigma_2,1))=\phi(i(\sigma_1,1)i(\sigma_2,1))=\phi(i(\sigma_1,1))\sigma_1(\phi(i(\sigma_2,1)))=\zeta(\sigma_1) \cdot \sigma_1(\zeta(\sigma_2)).$
    
    Similarly, we can easily show from the definition of $\gamma$ and the cocycle property of $\phi$ that $\gamma$ is a cocycle. It is a homomorphism because the group $\Gal(KK'/\Q)$ acts via $\Gal(K/\Q).$
    
    We know that \[\phi((\sigma,\tau))=\phi((\sigma,1))\cdot \sigma(\phi(1,\tau))=\zeta(\sigma) \cdot \sigma(\gamma(\tau))\] and \[\phi((\sigma,\tau))=\phi((1,\tau))\cdot \phi((\sigma,1))=\gamma(\tau)\cdot\zeta(\sigma).\]
    Combining these, we get that the image of $\gamma$ is contained in the set \[\{a \in A|~ \sigma(a)=\zeta(\sigma)^{-1}\cdot a\cdot \zeta(\sigma)~\forall~\sigma \in \Gal(K/\Q)\}.\]

\item We will show that $\phi$ satisfies the cocycle property. Let $(\sigma_1,\tau_1)$ and $(\sigma_2,\tau_2) \in \Gal(K/\Q) \times \Gal(K'/\Q).$ Then,

\begin{equation*}
\begin{aligned}
 \phi((\sigma_1,\tau_1)(\sigma_2,\tau_2))&= \zeta(\sigma_1\sigma_2) \cdot (\sigma_1\sigma_2)(\gamma(\tau_1\tau_2))\\
 &=\zeta(\sigma_1) \cdot \sigma_1 (\zeta(\sigma_2) \cdot \sigma_2(\gamma(\tau_1\tau_2)))\\
 &=\zeta(\sigma_1) \cdot \sigma_1 (\gamma(\tau_1)\gamma(\tau_2)\cdot \zeta(\sigma_2))\\
 &=\zeta(\sigma_1) \cdot \sigma_1(\gamma(\tau_1)) \cdot \sigma_1(\gamma(\tau_2)\cdot \zeta(\sigma_2))\\
 &=\phi((\sigma_1,\tau_1)) \cdot (\sigma_1,\tau_1)(\phi((\sigma_2,\tau_2))).
\end{aligned}
\end{equation*}
\end{romanenum}

\end{proof}

Fix a nice curve $X$ over $\Q$ with a nonconstant morphism $\pi_X \colon X \to \mathbb{P}^{1}_{\Q}.$ Let $K$ be a Galois extension of $\Q.$ Since $\pi_X$ is nonconstant, $\Aut_K(X,\pi_X)$ is a finite $\Gal(K/\Q)$-group.

We will now define a map $\theta$ between the set of $K$-twists of $(X,\pi_X)$ up to $\Q$-isomorphism and $H^{1}(\Gal(K/\Q),\Aut_{K}(X,\pi_X))$. 

Let $(Y,\pi_Y)$ be a $K$-twist of $(X,\pi_X).$ So, there is an isomorphism $f \colon X_K \to Y_K$ such that $\pi_{Y}\circ f = \pi_{X}.$ The map $\xi \colon \Gal(K/\Q) \to \Aut_{K}(X,\pi_X)$ defined as $\xi(\sigma)=f^{-1}\circ \sigma(f)$ is a cocycle. Using that $\pi_{Y}$ and $\pi_{X}$ are defined over $\Q$, we get that $\pi_{X} \circ \xi(\sigma)= \pi_{X}$ for every $\sigma \in \Gal(K/\Q).$ We define $\theta((Y,\pi_Y))=[\xi] \in H^{1}(\Gal(K/\Q),\Aut_{K}(X,\pi_X)).$ 

\begin{proposition}
The map $\theta$ defined above is well-defined and is a bijection between the set of $K$-twists of $(X,\pi_X)$ up to $\Q$-isomorphism and $H^{1}(\Gal(K/\Q),\Aut_{K}(X,\pi_X)).$ 
\end{proposition}

\begin{proof}
We first show that $\theta$ is well defined. Let $f_1 \colon X_K \to Y_K$ and $f_2 \colon X_K \to Y_K$ be two isomorphisms such that $\pi_{Y}\circ f_i = \pi_{X}$ for $i \in \{1,2\}$. For $i \in \{1,2\}$, let $\xi_i \colon \Gal(K/\Q) \to \Aut_K(X,\pi_X)$ be the cocycle $\xi_i(\sigma)=f_i^{-1}\circ \sigma(f_i)$. Then, it is easy to check that $\xi_1$ and $\xi_2$ are cohomologous with $a=f_1^{-1}\circ f_2$. Therefore, $\theta((Y,\pi_Y)) \in H^{1}(\Gal(K/\Q),\Aut_{K}(X,\pi_X))$ and does not depend on the choice of $f.$

Suppose $(Y,\pi_Y)$ and $(Y',\pi_{Y'})$ are $\Q$-isomorphic. There is an isomorphism $h \colon Y \to Y'$ such that $\pi_{Y'}\circ h = \pi_{Y}.$ Fix an isomorphism $f \colon X_K \to Y_K$ that satisfies $\pi_Y \circ f=\pi_X.$ Then $h \circ f \colon X_K \to Y'_K$ is an isomorphism satisfying $\pi_{Y'} \circ h \circ f=\pi_X.$ Since, $(h \circ f)^{-1} \circ \sigma(h \circ f)=f^{-1} \circ h^{-1} \circ \sigma(h) \circ \sigma(f)=f^{-1} \circ \sigma(f)$, we have $\theta((Y,\pi_Y))=\theta((Y',\pi_{Y'})).$ Therefore, given a $K$-twist of $X$, we get a class $[\xi]$ in $H^{1}(\Gal(K/\Q),\Aut_{K}(X,\pi_X))$ independent of $\Q$-isomorphism class.    

Let us show that the map is surjective. Consider any class $[\xi]$ in $H^{1}(\Gal(K/\Q),\Aut_{K}(X,\pi_X)).$ Using Proposition $1$, section $1$, Chapter $3$ in \cite{MR1466966} we get a nice curve $Y$ over $\Q$ with an isomorphism $f \colon X_K \to Y_K$ such that the cocycle $f^{-1} \circ \sigma(f)$ is cohomologous to $[\xi]$ in $H^{1}(\Gal(K/\Q),\Aut(X_K))$ i.e., there exists $A \in \Aut(X_K)$ such that $f^{-1} \circ \sigma(f)=A^{-1}\circ \xi(\sigma)\circ \sigma(A)$. Rearranging the terms we get $(f\circ A^{-1})^{-1} \circ \sigma(f\circ A^{-1})=\xi(\sigma).$ Since $\pi_{X} \circ \xi(\sigma)= \pi_{X}$, the morphism $\pi_{Y}:=\pi_{X} \circ (f\circ A^{-1})^{-1}$ is defined over $\Q$ because $\pi_Y= \sigma(\pi_Y)$ for every $\sigma$ in $\Gal(K/\Q).$ Therefore $(Y,\pi_Y)$ is a $K$-twist of $(X,\pi_X)$ that maps to $[\xi].$

Finally, let us show that the map is injective. Let $\xi_{1}$
and $\xi_{2}$ are two cocycles corresponding to $(Y,\pi_Y)$ and $(Y',\pi_{Y'})$ respectively, such that $\xi_{1}(\sigma)=A^{-1}\circ \xi_{2}(\sigma)\circ \sigma(A)$ with $A \in \Aut_{K}(X,\pi_X).$ Let $f \colon X_K \to Y_K$ be an isomorphism associated to $Y$ and $g \colon X_K \to Y'_K$ be an isomorphism associated to $Y'$, then $g \circ A \circ f^{-1} \colon Y \to Y'$ is an isomorphism defined over $\Q$ such that $\pi_{Y'}\circ g \circ A \circ f^{-1}=\pi_{Y}.$  
\end{proof}

Let $(Y,\pi_Y)$ be a $K$-twist of $(X,\pi_X)$. Fix an isomorphism $f \colon X_{K} \to Y_{K}$ that satisfies $\pi_Y \circ f=\pi_X.$ The cocycle $\zeta \colon \Gal(K/\Q) \to \Aut_{K}(X,\pi_X)$, $\sigma \mapsto f^{-1}\circ \sigma(f)$ is a representative of $\theta((Y,\pi_Y)).$

\begin{lemma}\label{Aut:twists}

We have $\Aut_{K}(Y,\pi_Y)=\{f \circ g \circ f^{-1}~|~ g \in \Aut_{K}(X,\pi_X)\}.$
\end{lemma}

\begin{proof}

Let $h \in \Aut_{K}(Y,\pi_Y).$ We can rewrite $h$ as $f\circ f^{-1} \circ h \circ f \circ f^{-1}$ and we have $\pi_X \circ f^{-1} \circ h \circ f =\pi_Y \circ h \circ f =\pi_Y \circ f =\pi_X.$ Therefore, $h=f \circ g \circ f^{-1}$ for some $g \in \Aut_{K}(X,\pi_X).$ We proved that $\Aut_{K}(Y,\pi_Y) \subseteq \{f \circ g \circ f^{-1}~|~ g \in \Aut_{K}(X,\pi_X)\}.$

If $g \in \Aut_{K}(X,\pi_X)$, then $\pi_Y \circ f \circ g \circ f^{-1}=\pi_X \circ g \circ f^{-1}=\pi_X \circ f^{-1}=\pi_Y.$ Therefore, $\{f \circ g \circ f^{-1}~|~ g \in \Aut_{K}(X,\pi_X)\} \subseteq \Aut_{K}(Y,\pi_Y).$
\end{proof}

\begin{lemma}\label{Automorph over Q}

We have \[\Aut(Y,\pi_Y)=\{f \circ g \circ f^{-1}~|~ g \in \Aut_{K}(X,\pi_X), \sigma(g)=\zeta(\sigma)^{-1}\circ g \circ \zeta(\sigma)~\forall~\sigma \in \Gal(K/\Q)\}.\]
\end{lemma}

\begin{proof}

An element $h \in \Aut_{K}(Y,\pi_Y)$ belongs to the group $\Aut(Y,\pi_Y)$ if and only if $\sigma(h)=h$ for all $\sigma \in \Gal(K/\Q).$ The lemma now follows using \cref{Aut:twists}.
\end{proof}

Now assume that the curve $X$ is $\PP^1_{\Q}.$ Let $(Y,\pi_Y)$ be a $K$-twist  of $(X,\pi_X).$ Fix an isomorphism $f \colon X_K \to Y_K$ and let $\zeta \colon \Gal(K/\Q) \to \Aut_K(X,\pi_X)$ be the cocycle $\zeta(\sigma)=f^{-1}\circ \sigma(f).$

\begin{lemma}\label{lemma:coboundary P^1}

The curve $Y$ is isomorphic to $\PP^1_{\Q}$ if and only if $\zeta \colon \Gal(K/\Q) \to \Aut_K(X,\pi_X) \injects \Aut(X_K)$ is a coboundary in $H^{1}(\Gal(K/\Q),\Aut(X_K)).$
\end{lemma}

\begin{proof}

Assume that $Y$ is isomorphic to $\PP^1_{\Q}.$ Then there exists an isomorphism $h \colon X \to Y.$ Since, $h$ is defined over $\Q$ we have $\zeta(\sigma)=f^{-1} \circ \sigma(f)=(h^{-1}\circ f)^{-1}\circ \sigma(h^{-1} \circ f).$ Since, $h^{-1} \circ f \in \Aut(X_K)$ the cocycle $\zeta$ is a coboundary in $H^{1}(\Gal(K/\Q),\Aut(X_K)).$

Suppose $\zeta$ is a coboundary in $H^{1}(\Gal(K/\Q),\Aut(X_K))$ i.e., $\zeta(\sigma)=f^{-1}\circ \sigma(f)=A^{-1}\sigma(A)$ for some $A \in \Aut(X_K).$ Therefore, $f \circ A^{-1} \colon X \to Y$. Hence, $Y$ is isomorphic to $\PP^1_{\Q}.$ 
\end{proof}

\begin{remark}\label{rmk:application P^1}

In our application, given a nice curve $(X,\pi_X)$ where $X$ is $\PP^1_{\Q}$ we need to find $K$-twists $(Y,\pi_Y)$ such that $Y$ is $\PP^1_{\Q}.$ We do this as follows. We consider the cocycles from $\Gal(K/\Q)$ to $\Aut_K(X,\pi_X).$ There are finitely many if $\Gal(K/\Q)$ is finite. We check if a cocycle \[\zeta \colon \Gal(K/\Q) \to \Aut_K(X,\pi_X) \injects \Aut(X_K)=\PGL_2(K) \] is a coboundary in $H^{1}(\Gal(K/\Q),\Aut(X_K)).$ If $\zeta$ is a coboundary in $H^{1}(\Gal(K/\Q),\PGL_2(K))$, then there is an $A \in \PGL_2(K)$ such that  $\zeta(\sigma)=A^{-1}\sigma(A)$ for every $\sigma \in \Gal(K/\Q)$; $A$ is an isomorphism between $X_K \to \PP^1_{\Q}$ and using $\pi_X \circ \zeta(\sigma)=\pi_X$ for every $\sigma \in \Gal(K/\Q)$ we observe that $\pi_X \circ A^{-1} \colon \PP^1_{\Q} \to \PP^{1}_{\Q}.$ The pair $(\PP^1_{\Q},\pi_X \circ A^{-1})$ is twist of $(X,\pi_X).$
\end{remark}

\subsection{Determining if a cocycle is a coboundary}

Let $K$ be a finite Galois extension of $\Q$ and let \[\zeta \colon \Gal(K/\Q) \to \Aut(\PP^1_{K}) = \PGL_{2}(K)\] be a cocycle. The cocycle $\zeta$ determines a twist $C$ of $\PP^1_{\Q}$ up to isomorphism. In this section, we explain how to explicitly compute $C$ as a conic in $\PP^2_{\Q}.$ 
The Hasse principle can then be used to determine if $C(\Q)$ is empty or not. If $C(\Q)$ is nonempty, then it is isomorphic to $\PP^1_{\Q}$ and equivalently $\zeta$ is a coboundary.
When $\zeta$ is a coboundary, we will explain how to compute an $A \in \PGL_{2}(K)$ such that $\zeta(\sigma)=A^{-1}\sigma(A)$ for every $\sigma \in \Gal(K/\Q).$

Let $Q_0$ be the quadratic form $y^2-xz.$ Let $C_{0}$ be the conic defined by $Q_0=0$ in $\mathbb{P}^{2}_{\Q}$. 

Let $\phi \colon \PGL_2(K) \to \GL_3(K)$ be the map \[[(\begin{smallmatrix}
a & b \\
c & d
\end{smallmatrix})] \mapsto \frac{1}{ad-bc}\bigg(\begin{smallmatrix}
a^{2} & 2ab & b^{2} \\
ac & ad+bc & bd \\
c^{2} & 2cd & d^{2}    
\end{smallmatrix}\bigg).\] The map $\phi$ comes from remark 3.3 in \cite{MR3906177}. It is an injective group homomorphism that respects the $\Gal(K/\Q)$-action. The image of $\phi$ gives automorphisms of $\PP^2_K$ that stabilizes $(C_0)_K.$ The map $\phi$ induces an isomorphism $\Bar{\phi} : \Aut(\PP^{1}_{K})=\PGL_2(K) \to \Aut((C_0)_{K})$ that respects the $\Gal(K/\Q)$- action. One can check that $\Bar{\phi}$ is induced by the isomorphism $\PP^1_{\Q} \to C_0$ that sends $[x:y]$ to $[x^2:xy:y^2].$

Define the map \[\bar{\zeta}:=\phi \circ \zeta \colon \Gal(K/\Q) \to \GL_3(K).\] Since $\phi$ is a group homomorphism and respects the $\Gal(K/\Q)$-action, $\Bar{\zeta}: \Gal(K/\Q) \to \GL_3(K)$ is a cocycle. By Hilbert 90, there is an $M \in \GL_{3}(K)$ such that $\Bar{\zeta}(\sigma)=M^{-1}\sigma(M)$ for every $\sigma \in \Gal(K/\Q).$ We explain how to compute $M$ in \cref{Computing H90 mat}.

\begin{lemma}\label{twist as a conic}

The quadratic form $Q:=Q_{0}(M^{-1}(x,y,z)^{T})$ has coefficients in $\Q$ and the conic $C \subseteq \PP^2_{\Q}$ defined by $Q=0$ is isomorphic to the twist of $\PP^1_{\Q}$ by $\zeta.$
\end{lemma}

\begin{proof}
Since the cocycle $\Bar{\zeta}(\sigma)$ preserves the equation for $Q_{0}$ for every $\sigma \in \Gal(K/\Q)$, it follows that $Q$ has coefficients in $\Q$. 
\end{proof}

From \cref{twist as a conic}, $M$ gives a conic $C \subseteq \PP^2_{\Q}$ which is a twist of $C_0$ by $\zeta.$ The cocycle $\zeta$ is a coboundary if and only if $C$ has a rational point. We use \texttt{Magma} \cite{MR1484478} to check this. Suppose that the conic $C$ has a rational point. There is a change of variable matrix $B \in \GL_3(\Q)$ from $C$ to $C_{0}$. Observe that $BM \in \Aut_K(C_0)$ and $\bar{\zeta}(\sigma)=M^{-1}\sigma(M)=(BM)^{-1}\sigma(BM).$ Define $A:=(\bar{\phi})^{-1}(BM) \in \PGL_2(K).$ One can then check that $\zeta(\sigma)=A^{-1}\sigma(A)$ for every $\sigma \in \Gal(K/\Q).$

\subsection{Computing Hilbert 90 matrices}\label{Computing H90 mat}

Let $K$ be a finite Galois extension of $\Q$ and define $G:= \Gal(K/\Q)$. Hilbert 90 states that $H^{1}(G,\GL_{n}(K))=1.$ Equivalently, given a cocycle $\psi \colon G \to \GL_{n}(K)$ there exists a matrix $A\in \GL_{n}(K)$ such that $\psi(\sigma)=A^{-1}\sigma(A)$ for every $\sigma \in G.$ We now describe how to find $A$.

\begin{itemize}
    \item Define the $K$-vector space $V:=K^n.$ Define a new $G$-action on $V$ by $\sigma*v=\psi(\sigma)\sigma(v)$; it acts $\Q$-linearly. Let $W$ be the $\Q$-subspace of $V$ fixed by this $G$-action. Using Proposition $7 (A.V.63)$ of \cite{MR1994218}, the natural map $K \tensor_{\Q} W \to V$ of $K$-vector spaces is an isomorphism. In particular, $W$ has dimension $n$ over $\Q$. 
    
    \item Fix a basis $\frakA$ for $K$ considered as a vector space over $\Q$ and a basis $\frakB$ for $V$ considered as a vector space over $K$. The set $B=\{b.v~|~b \in \frakA, v \in \frakB\}$ is a basis of $V$ over $\Q.$
    
    \item The $\Q$-linear map $V\to W$ that sends $v$ to $1/|G|(\Sigma_{g \in G} g*v)$ is a projection. Therefore, the set $C =\{1/|G|(\Sigma_{g \in G} g*u) : u \in B\}$ spans $W$ over $\Q$.
    
    \item Choose $n$ linearly independent vectors from $C$; they form a basis for $W$ and hence also a basis of $V$ over $K.$ Let $A \in \GL_n(K)$ be the change of basis matrix from our basis in C to the standard basis of $K^n$. We have $\psi(\sigma)=A^{-1}\sigma(A)$ for all $\sigma \in G.$  
\end{itemize}

\section{A description of groups associated to abelian family of twists}\label{Horz group decsription}

Let $(X_G,\pi_G)$ be a modular curve. Let $\calA$ be a finite abelian subgroup of $\Aut(X_G,\pi_G).$ Given a cocycle $\gamma: \Gal(\Qab/\Q) \to \calA$ which is surjective we get a $\Qab$-twist $(X_{G_{\gamma}},\pi_{G_{\gamma}})$ of $(X_G,\pi_G).$ In this section we give a description of the group $G_{\gamma}.$ Let $\Gamma$ be $G \intersect \SL_2(\ZZ).$

The morphism $\pi_G \colon X_G \to \PP^1_{\Q}$ corresponds to the inclusion of fields $\Q(j) \subseteq \Q(X_G).$ The group $\calA$ acts on the function field $\Q(X_G)$ on the right and fixes $\Q(j).$ Let $\Q(X_G)^{\calA}$ be the subfield of $\Q(X_G)$ fixed by $\calA$ and $X_G^{\calA}$ be the nice curve over $\Q$ whose function field is $\Q(X_G)^{\calA}.$ The inclusion of fields $\Q(j) \subseteq \Q(X_G)^{\calA} \subseteq \Q(X_G)$ corresponds to morphisms $\pi \colon X_G \to X_G^{\calA}$ and $\pi^{'}_{G} \colon X_G^{\calA} \to \PP^1_{\Q}$ satisfying $\pi^{'}_{G} \circ \pi=\pi_G.$ Then, $(X_G^{\calA},\pi^{'}_{G})$ is a modular curve $(X_{G_0},\pi_{G_0})$ for an open subgroup $G_0$ of $\GL_2(\Zhat)$ satisfying $-I \in G_0$ and $\det(G_0)=\Zhat^{\times}.$  

The extension $ \Q(X_G)^{\calA} \subseteq \Q(X_G)$ is Galois with $\calA^{op}$ as Galois group. So, $\Aut(X_G,\pi)=\calA.$ Since, $\pi$ is a Galois cover $G$ is a normal subgroup of $G_0$ and we have a natural isomorphism $G_0/G \isom \calA.$ We also observe that $\pi_{G_{\gamma}}$ factors through $X_{G_0}.$ We will give a description of the group $G_{\gamma}$ in terms of $G_0$ and $\gamma.$

For $\sigma \in \Gal(\Q^{cyc}/\Q)$ and $f \in \Q^{cyc}(X_G)$, let $\sigma(f)$ denote the action of $\Gal(\Q^{cyc}/\Q)$ on $\Q^{cyc}(X_G)$ via the field $\Q^{cyc}$, i.e., for every $f \in \Q(X_G)$ and for every $\sigma \in \Gal(\Q^{cyc}/\Q)$ we have $\sigma(f)=f.$ Let us define a new action of $\Gal(\Q^{cyc}/\Q)$ on $\Q^{cyc}(X_G)$ as follows. For $\sigma \in \Gal(\Q^{cyc}/\Q)$ and $f \in \Q^{cyc}(X_G)$, let $\sigma \bullet f=\sigma(f) \circ \gamma_{\sigma}^{-1}.$
 
Then, \[\Q(X_{G_{\gamma}})=\{f \in \Q^{cyc}(X_G)| \sigma \bullet f=f~\forall \sigma \in \Gal(\Q^{cyc}/\Q)\}.\] 

Since the cyclotomic character $\chi_{cyc} \colon \Gal(\Qab/\Q) \to \Zhat^{\times}$ is an isomorphism, there is an unique homomorphism $\phi_{\gamma}$ that makes the following diagram commutative.

\begin{displaymath}
  \xymatrix{
    {\Gal(\Qab/\Q)} \ar[rr]^{\chi_{cyc}} \ar[dr]_{\gamma^{-1}}
    && {\Zhat^{\times}} \ar[dl]^{\phi_{\gamma}}\\
    & {G_0/G}}
\end{displaymath}

Define \[S:=\{g \in G_0 |gG=\phi_{\gamma}(\det g)\}.\]   

The action $\bullet$ of $\Gal(\Q^{cyc}/\Q)$ and $*$ of $\G_{\gamma}/\Gamma$ on $\Q^{\cyc}(X_G)$ are compatible with respect to the isomorphism $\Gal(\Q^{cyc}/\Q) \xrightarrow[\chi_{cyc}]{\simeq} \Zhat^{\times} \xleftarrow[\det]{\simeq} G_{\gamma}/\Gamma.$ Explicitly, if $\chi_{\cyc}(\sigma)=\det(g)$ with $\sigma \in \Gal(\Q^{cyc}/\Q)$ and $g \in G_{\gamma}$, then $\sigma \bullet f=f*g$ for all $f \in \Q^{cyc}(X_G).$ 

\begin{lemma} \label{G_{gamma}}

The group $G_{\gamma}=S.$
\end{lemma}

\begin{proof}

Let us show that $G_{\gamma} \subseteq S.$ 
Let $g \in G_{\gamma}.$ Let $\sigma \in \Gal(\Q^{cyc}/\Q)$ satisfy $\chi_{\cyc}(\sigma)=\det(g).$ For every $f \in \Q(X_G)$, $f*g=\sigma \bullet f = \sigma(f) \circ \gamma_{\sigma}^{-1}=f \circ \gamma_{\sigma}^{-1}.$ Therefore, $gG= \phi_{\gamma}(\det g).$

We will now show that $S \subseteq G_{\gamma}.$
Let $g \in S$. Let $\sigma \in \Gal(\Q^{cyc}/\Q)$ satisfy $\chi_{\cyc}(\sigma)=\det(g)$, i.e., $gG=\gamma_{\sigma}^{-1}.$ For any $f \in \Q(X_{G_{\gamma}})$ we have $f=\sigma(f) \circ \gamma_{\sigma}^{-1}=f*g.$ Therefore, $g \in G_{\gamma}.$
\end{proof}

\section{Abelian Families of Twists}\label{sec:families of twists}

Fix a finite abelian subgroup $\calA \subseteq \PGL_2(\Q)=\Aut(\PP^1_{\Q}).$ We choose a function $\pi(t) \in \Q(t)$ satisfying $\Q(t)^{\calA}=\Q(\pi(t)).$ The inclusion $\Q(\pi(t)) \subseteq \Q(t)$ gives a morphism $\pi : \PP^1_{\Q} \to \PP^1_{\Q}$ satisfying $\Aut_{\Qbar}(\PP^1_{\Q},\pi)=\calA.$

Fix a variable $v$ and let $x$ be a root of the polynomial $\pi(T)-v \in \Q(v)[T].$ Let $K:=\Q(x).$ The field extension $\Q(v) \subseteq K$ is Galois. 

For each $\sigma \in \Gal(K/\Q(v))$, there is a unique $\gamma_{\sigma} \in \calA$ satisfying \[\sigma(x)=\gamma_{\sigma}(x).\] It can be checked that $\gamma : \Gal(K/\Q(v)) \stackrel{\sim}{\to} \calA \subseteq \PGL_2(\Q) \subseteq \PGL_2(K)$ is a homomorphism. 

In the cases we consider, we find a matrix $A \in \GL_2(K)$ such that $\gamma_{\sigma}=A^{-1} \sigma(A)$ for every $\sigma \in \Gal(K/\Q(v)).$ This matrix $A$ describes an isomorphism $f:\PP^1_K \to \PP^1_K$ such that $\pi \circ f^{-1}$ is stable under $\Gal(K/\Q(v))$-action. Define $\pi_{\calA,v} := \pi \circ f^{-1} \in \Q(v)(u).$ 

\subsection{Specialization}

Here $v$ is a variable that can be specialized at rational values. More precisely, we will take $v \in \Q \subseteq \Q \union \{\infty\} \subseteq \PP^1(\Q)$ for which $\pi$ is unramified at $v.$ The cocycle $\gamma$ gives a $K$-twist of $(\PP^1_{\Q},\pi)$ that has a rational point. Moreover, any $\Qab$-twist of $(\PP^1_{\Q},\pi)$ that has a rational point arises from such a $v \in \Q.$

Given $v \in \Q$, we find the corresponding twist $(\PP^1_{\Q},\pi_{\calA,v})$ where we can take $\pi_{\calA,v} \in \Q(u).$

We are going to limit our discussion to those subgroups $\calA$ that show up when $(\PP^1_{\Q},\pi)$ is a modular curve. As a remark, we point out that there are finite abelian subgroups of $\PGL_2(\Q)$ which do not show up in our context for example, the cyclic group of order $6.$ Please see \cite{MR2681719} for more details on classification of finite subgroups of $\PGL_2(\Q).$

For each $\calA$ we give the function $\pi$, a matrix $A$ and resulting $\pi_{\calA,v}.$

\begin{itemize}

\item $\mathbf{Case\ 0}$ For $\calA =\{(\begin{smallmatrix}
1 & 0 \\ 
0  & 1
\end{smallmatrix})\}$, the function $\pi(T)=T.$ A matrix $A=(\begin{smallmatrix}
1 & 0 \\ 
0  & 1
\end{smallmatrix}) $ and $\pi_{\calA,v}(u)=u.$

\item $\mathbf{Case\ 1}$ For $\calA =\{(\begin{smallmatrix}
1 & 0 \\ 
0  & 1
\end{smallmatrix}),(\begin{smallmatrix}
-1 & 0 \\ 
0  & 1
\end{smallmatrix})\}$, the function $\pi(T)=T^2.$ 
A matrix $A=(\begin{smallmatrix}
1 & 0 \\ 
0  & x
\end{smallmatrix}) $ and $\pi_{\calA,v}(u)=vu^2.$

\item $\mathbf{Case \ 2}$ For $\calA =\{(\begin{smallmatrix}
1 & 0 \\ 
0  & 1
\end{smallmatrix}),(\begin{smallmatrix}
0 & \alpha \\ 
1  & 0
\end{smallmatrix})\}$, where $\alpha$ is a non-zero rational number the function $\pi(T)= T+\alpha/T.$ 
A matrix $A=(\begin{smallmatrix}
x & \alpha \\ 
1  & x
\end{smallmatrix}) $ and $\pi_{\calA,v}(u)=(vu^{2}-4\alpha u+\alpha v)/(-u^{2}+vu-\alpha).$

\item $\mathbf{Case \ 3}$ For $\calA =\{(\begin{smallmatrix}
1 & 0 \\ 
0  & 1
\end{smallmatrix}),(\begin{smallmatrix}
0 & -1 \\ 
1  & -1
\end{smallmatrix}),(\begin{smallmatrix}
1 & -1 \\ 
1  & 0
\end{smallmatrix})\},$ the function $\pi(T)=(T^3-3T+1)/(T^2-T).$ 
A matrix \[A=\Big(\begin{smallmatrix}
(-6v+3)x^{2} + (6v^{2}-6v+6)x +
    (-3v^{2}+12v-6)  & ~~~(3v-6)x^{2} +
    (-3v^{2}+3v-3)x + (3v^{2}-9v+ 12)\\ 
6x^{2} + (-6v + 3)x + (3v - 12)  & -3x^{2} + (3v + 3)x +
    (-3v + 6)
\end{smallmatrix}\Big)\] and
$\pi_{\calA,v}(u)=((-v + 3)u^{3} + (-3v^{2} + 9v - 9)u^{2} + (-3v^{3} + 9v^{2} - 15v)u + (-v^{4} + 3v^{3} - 6v^{2} - v + 3))/(u^{3} + 2vu^{2} + (v^{2} + v - 3)u + (v^{2} - 3v + 1)).$

\item $\mathbf{Case \ 4}$ For $\calA =\{(\begin{smallmatrix}
1 & 0 \\ 
0  & 1
\end{smallmatrix}),(\begin{smallmatrix}
0 & -1 \\ 
1  & 0
\end{smallmatrix}),(\begin{smallmatrix}
-1 & -1 \\ 
1  & -1
\end{smallmatrix}),(\begin{smallmatrix}
1 & -1 \\ 
1  & 1
\end{smallmatrix})\},$ the function \\ $\pi(T)=(T^{4}-6T^{2}+1)/(T^3-T).$ 
 
A matrix \[A=\Big(\begin{smallmatrix}
(-2v + 24)x^{3} + (2v^{2} - 22v)x^{2} + (-2v^{2} +
    10v - 136)x + (10v + 16)
  & (4v - 8)x^{3} + (-4v^{2} + 4v)x^{2} + (4v^{2} - 20v
    + 72)x - 32\\ 
8x^{3} - (8v)x^{2} - 40x + 4v
  & (2v - 8)x^{3} + (-2v^{2} + 6v)x^{2} + (2v^{2} - 10v
    + 56)x + (-2v - 16)
\end{smallmatrix}\Big)\] and $\pi_{\calA,v}(u)=(-vu^{4} + (8v + 16)u^{3} + (-18v - 96)u^{2} + (8v + 176)u + (7v - 96))/(u^{4} +(v - 8)u^{3} + (-6v + 18)u^{2} + (11v - 8)u + (-6v - 7)).$

\item $\mathbf{Case \ 5}$ For $\calA=\{(\begin{smallmatrix}
1 & 0 \\ 
0  & 1
\end{smallmatrix}),(\begin{smallmatrix}
0 & \alpha \\ 
1  & 0
\end{smallmatrix}),(\begin{smallmatrix}
-1 & 0 \\ 
0  & 1
\end{smallmatrix}),(\begin{smallmatrix}
0 & -\alpha \\ 
1  & 0
\end{smallmatrix})\},$ where $\alpha$ is a non-zero rational number the function $\pi(T) = T^{2}+\alpha^2/T^2.$ 
 
A matrix $A=(\begin{smallmatrix}
1
  & x\\ 
x^2
  & \alpha^2/x
\end{smallmatrix})$ and $\pi_{\calA,v}(u)=((-3v\alpha^2 + v^3)\alpha^2u^4 + (8\alpha^2 - 4v^2)\alpha^2u^3 + 6v\alpha^2u^2 - 8\alpha^2u +v)/(\alpha^4u^4 - 2v\alpha^2u^3 + (2\alpha^2 + v^2)u^2 - 2vu + 1).$

\item $\mathbf{Case \ 6}$ For $\calA=\{(\begin{smallmatrix}
1 & 0 \\ 
0  & 1
\end{smallmatrix}),(\begin{smallmatrix}
0 & -1 \\ 
1  & 0
\end{smallmatrix}),(\begin{smallmatrix}
1 & 1 \\ 
1  & -1
\end{smallmatrix}),(\begin{smallmatrix}
-1 & 1 \\ 
1  & 1
\end{smallmatrix})\},$ the rational function \\
$\pi(T)=(T^4 + 2T^2 + 1)/(T^3 - T).$ 
 
A matrix $A=\Big(\begin{smallmatrix}
a_1  & a_2\\ 
a_3  & a_4 \end{smallmatrix}\Big)$, where
    $a_1=\sfrac{8x^3}{(v^2-16)} - \sfrac{8vx^2}{(v^2-16)} + \sfrac{24x}{(v^2-16)} + \sfrac{4v}{(v^2-16)},$\\ $a_2=\sfrac{-2x^3}{(v-4)} + \sfrac{(2v^2 + 10v)x^2}{(v^2 - 16)} + \sfrac{(-2v^2 - 6v - 8)x}{(v^2-16)} + \sfrac{-6v-16}{(v^2-16)},$\\ $a_3=\sfrac{(-2v + 16)x^3}{(v^2-16)} + \sfrac{(2v^3 - 14v^2 - 32)x^2}{(v^3 - 16v)} +
    \sfrac{(-2v^2 - 6v + 80)x}{(v^2-16)} + \sfrac{(10v^2 - 16v - 32)}{(v^3-16v)}$ and \\ $a_4=\sfrac{(-4v + 8)x^3}{(v^2-4v)} + \sfrac{(4v^2 + 12v - 40)x^2}{(v^2-16)} +
    \sfrac{(-4v^3 - 4v^2 + 8v + 32)x}{(v^3-16v)} + \sfrac{-8v-8}{(v^2-16)}$ and $\pi_{\calA,v}(u)$ is \[\dfrac{\small{(-25v^3 + 160v^2 - 256v)u^4 + (40v^3 - 208v^2 + 256v)u^3+ (-26v^3 + 96v^2 - 64v)u^2+(8v^3 - 16v^2)u-v^3}}{\small{(6v^3 - 37v^2 + 64v - 64)u^4+(-11v^3 + 56v^2 - 32v)u^3+ (6v^3 - 30v^2)u^2+(-v^3 +8v^2)u-v^2}}.\]

\item $\mathbf{Case \ 7}$ For $\calA=\{(\begin{smallmatrix}
1 & 0 \\ 
0  & 1
\end{smallmatrix}),(\begin{smallmatrix}
0 & -5 \\ 
1  & 0
\end{smallmatrix}),(\begin{smallmatrix}
-1 & 5 \\ 
1  & 1
\end{smallmatrix}),(\begin{smallmatrix}
5 & 5 \\ 
1  & -5
\end{smallmatrix})\},$ the function \\
$\pi(T)=(T^4 + 10T^2 + 25)/(T^3 - 4T^2- 5T).$ 
 
A matrix $A=\Big(\begin{smallmatrix}
a_1  & a_2\\ 
a_3  & a_4 \end{smallmatrix}\Big)$, where
 $a_1=\sfrac{(22v^2/15 - 84v/5 - 224/3)x^3}{(v^3 - 36v^2 + 240v + 1600)}\\ + \sfrac{(-22v^3/15
    + 242v^2/15 + 248v/3 + 80/3)x^2}{(v^3 - 36v^2 + 240v + 1600)}\\ + \sfrac{(98v^3/15 - 898v^2/15 - 1516v/3 - 2080/3)x}{(v^3 - 36v^2 + 240v + 1600)} + \sfrac{(4v^3 - 142v^2/3 - 200v + 400/3)}{(v^3 - 36v^2 + 240v + 1600)}$, 
    $a_2=\sfrac{(2v^2/3 - 8v - 80/3)x^3}{(v^3 - 36v^2 + 240v + 1600)} + \sfrac{ (-2v^3/3 +
    22v^2/3 + 92v/3 + 320/3)x^2}{(v^3 - 36v^2 + 240v + 1600)} + \sfrac{(10v^3/3
    - 98v^2/3 - 760v/3 - 400/3)x}{(v^3 - 36v^2 + 240v + 1600)} + \sfrac{ (50v^2/3 - 60v - 3200/3)}{(v^3 - 36v^2 + 240v + 1600)},$\\
    $a_3=\sfrac{(-22v^2/15 + 152v/15 + 32)x^3}{(v^4 - 36v^3 + 240v^2 + 1600v)} + \sfrac{(22v^2/15 - 142v/15 - 112/3)x^2}{(v^3 - 36v^2 + 240v + 1600)} +
    \sfrac{(-98v^3/15 + 458v^2/15 + 728v/3 + 160)x}{(v^4 - 36v^3 + 240v^2 +
    1600v)} + \sfrac{ (-4v^2 + 82v/3 + 400/3)}{(v^3 - 36v^2 + 240v + 1600)}$ and \\ 
    $a_4=\sfrac{(-2v/3 + 16/3)x^3}{(v^3 - 36v^2 + 240v + 1600)} + \sfrac{(2v^3/3 - 14v^2/3 +
    8v/3 - 160)x^2}{(v^4 - 36v^3 + 240v^2 + 1600v)} \\
    +\sfrac{(-10v^2/3 + 46v/3+ 560/3)x}{(v^3 - 36v^2 + 240v + 1600)} + \sfrac{(10v^2/3 - 1400v/3 -
    800)}{(v^4 - 36v^3 + 240v^2 + 1600v)}$ 
    
    and $\pi_{\calA,v}(u)$ is $((-v^5 + 32v^4 - 336v^3 + 1280v^2 - 1600v)u^4 + (-4v^6 + 148v^5 - 1744v^4 + 5600v^3 + 16000v^2 -64000v)u^3+ (-6v^7 + 252v^6- 3376v^5 + 10880v^4 + 69600v^3 - 256000v^2 - 640000v)u^2+(-4v^8 + 188v^7 - 2864v^6 + 10680v^5 +82400v^4 - 368000v^3 - 1280000v^2)u+(-v^9 + 52v^8 - 896v^7 + 4120v^6 + 29500v^5 - 176000v^4 -640000v^3)
    )/((v^4 - 36v^3 + 240v^2 + 1600v - 14400)u^4+(3v^5 - 128v^4 + 1320v^3 + 1920v^2 - 51200v)u^3+ (3v^6 - 149v^5 + 2056v^4 - 3600v^3 -63200v^2 + 64000v)u^2+ (v^7
    - 58v^6 + 1002v^5 - 4400v^4 - 20800v^3 + 96000v^2)u+(-v^7 + 26v^6 - 280v^5 + 3500v^4 -
    16000v^3 - 160000v^2)).$

\item $\mathbf{Case \ 8}$ For $\calA=\{(\begin{smallmatrix}
1 & 0 \\ 
0  & 1
\end{smallmatrix}),(\begin{smallmatrix}
-2 & 2 \\ 
1  & 2
\end{smallmatrix}),(\begin{smallmatrix}
1 & 2 \\ 
1  & -1
\end{smallmatrix}),(\begin{smallmatrix}
0 & -2 \\ 
1  & 0
\end{smallmatrix})\},$ the rational function $\pi(T)=(T^4 + 4T^2 + 4)/(T^3 +T^2- 2T).$ 
 
A matrix $A=\Big(\begin{smallmatrix}
a_1  & a_2\\ 
a_3  & a_4 \end{smallmatrix}\Big)$,
where
    $a_1=\sfrac{(-v^3/9 + 10v^2/9 + 28v/9 - 64/9)x^3}{(v^3 + 4v^2 - 32v)}\\
    + \sfrac{ (v^4/9 -
    14v^3/9 - 14v^2/3 + 104v/9 + 32/9)x^2}{(v^3 + 4v^2 - 32v)} + \sfrac{(5v^4/9 +
    2v^3/9 - 16v^2/3 + 88v/9 - 128/9)x}{(v^3 + 4v^2 - 32v)} + \sfrac{(28v^3/9 +
    44v^2/9 - 352v/9 + 64/9)}{(v^3 + 4v^2 - 32v)},$\\ 
    $a_2=\sfrac{(4v^3/9 + 14v^2/9 - 40v/9 - 32/9)x^3}{(v^3 + 4v^2 - 32v)} + \sfrac{ (-4v^4/9 - 16v^3/9 + 20v^2/3 + 88v/9 - 128/9)x^2}{(v^3 + 4v^2 - 32v)} + \sfrac{(-2v^3/9 +
    2v^2/9 + 92v/9 + 16/9)x}{(v^2 +8v)} + \sfrac{(-4v^3/9 + 40v^2/9 + 112v/9 - 256/9)}{(v^3 + 4v^2 - 32v)},$\\
    $a_3=\sfrac{(2v/9 - 32/9)x^3}{(v^3 + 4v^2 - 32v)} + \sfrac{(-2v^2/9 + 40v/9 + 16/9)x^2}{(v^3 + 4v^2 - 32v)} +
    \sfrac{(-10v^2/9 + 20v/9 - 64/9)x}{(v^3 + 4v^2 - 32v)} + \sfrac{  (-56v/9 + 32/9)}{(v^3 + 4v^2 - 32v)}$ and \\ 
    $a_4=\sfrac{(-8v/9 - 16/9)x^3}{(v^3 + 4v^2 - 32v)} + \sfrac{(8v^2/9 + 20v/9 - 64/9)x^2}{(v^3 + 4v^2 - 32v)} +
    \sfrac{(4v/9 + 8/9)x}{(v^2 +8v)} + \sfrac{(8v/9 - 128/9)}{(v^3 + 4v^2 - 32v)}$ 
    
   and $\pi_{\calA,v}(u)$ is $(vu^4 + (2v^3 + 4v^2 - 24v)u^3 + (3v^5/2 + 6v^4 - 29v^3 - 68v^2 +
    184v)u^2 + (v^7/2 + 3v^6 - 11v^5 - 62v^4 + 108v^3 + 320v^2 -
    480v)u + (v^9/16 + v^8/2 - 5v^7/4 - 14v^6 + 41v^5/4 + 130v^4 -
    80v^3 - 400v^2 + 400v))/(u^4 + (3v^2/2 + 2v - 8)u^3 + (3v^4/4 +
    7v^3/4 - 5v^2 - 8v + 24)u^2 + (v^6/8 + v^5/4 + v^4/4 + 3v^3 -
    10v^2 + 8v - 32)u + (-v^7/16 + 3v^6/8 + 7v^5/2 - 29v^4/4 - 31v^3 + 64v^2 + 16)).$
\end{itemize}

\section{Proof of \cref{THM:FIRST}}\label{proof 1}

Take a genus $0$ congruence subgroup $\Gamma$ of $\SL_2(\ZZ)$ that contains $-I.$ Let $N$ be the level of $\Gamma.$ In this section, we will prove that there are nonconstant $\pi_1,\ldots,\pi_l \in \Q(t)$ and finite abelian subgroups $\calA_i \subseteq \Aut(\PP^1_{\Q},\pi_i) \subseteq \PGL_2(\Q)$ such that the following hold:
\begin{romanenum}
    
    \item For each $i \in \{1,\ldots,l\}$, $(\PP^1_{\Q},\pi_i)$ is isomorphic to some modular curve $(X_G,\pi_G)$ where $G$ is an open subgroup of $\GL_2(\Zhat)$ satisfying $-I \in G$, $\det(G)=\Zhat^{\times}$ and $G \intersect \SL_2(\ZZ)=\Gamma.$
    
    \item For any modular curve $X_G$ isomorphic to $\PP^1_{\Q}$ where $-I \in G$, $\det(G)=\Zhat^{\times}$ and $G \intersect \SL_2(\ZZ)=\Gamma$, there is an $i \in \{1,\ldots,l\}$, such that $(X_G,\pi_G)$ is a $\Qab$-twist of $(\PP^1_{\Q},\pi_i)$ via a cocycle $\gamma \colon \Gal(\Qab/\Q) \to \calA_i \subseteq \Aut(\PP^1_{\Q},\pi_i).$   
\end{romanenum}   
It suffices to prove \cref{THM:FIRST} for a fixed $\Gamma$ because there are only finitely many genus $0$ congruence subgroups up to conjugation in $\GL_2(\ZZ).$

\subsection{Computing a pair $(\PP^1_{\Q},\pi_1)$}\label{computing a pair}

In this section, we will discuss our computation of a pair $(\PP^1_{\Q},\pi_1)$ that is isomorphic to a modular curve $(X_{G_1},\pi_{G_1})$ for some open subgroup $G_1$ of $\GL_2(\Zhat)$ satisfying $-I \in G_1$, $\det(G_1)=\Zhat^{\times}$ and $G_1 \intersect \SL_2(\ZZ)=\Gamma.$ We either find one such pair $(\PP^1_{\Q},\pi_1)$ or show that there does not exist any.

\begin{remark}
The point of computing such a pair is that to compute all the other modular curves $(X_G,\pi_G)$ such that $G \intersect \SL_2(\ZZ)$ is $\Gamma$ and the curve $X_{G}$ is isomorphic to $\PP^1_{\Q}$ we have to consider the cocycles $\phi \colon \Gal(\Qab/\Q) \to \Aut_{\Qab}(X_{G_1},\pi_{G_1}).$
\end{remark}

Let $h$ be the normalized hauptmodul of $\Gamma$. Since the coefficients of $q$-expansion of $h$ lie in $K_N$, the function field of $X_{\Gamma}$ is $K_N(h).$ There is a unique $\pi_{\Gamma}$ in $K_N(t)$ such that $\pi_{\Gamma}(h)=j$ where $j$ is the modular $j$-invariant. To compute $\pi_{\Gamma}$ we use the method described in section $4.4$ of \cite{MR3671434}; this uses that from \cref{sec:Haupt}, we can compute arbitrarily many terms of the $q$-expansion of $h.$ The function $\pi_{\Gamma}$ describes a morphism $\pi_{\Gamma} \colon X_{\Gamma} \to \mathbb{P}^1_{K_N}.$ 

Suppose there is an open subgroup $G$ of $\GL_2(\Zhat)$ containing $-I$ with full determinant for which $G \intersect \SL_2(\ZZ)$ is $\Gamma.$ Let $M$ be the level of $G.$ We know that $N$ divides $M.$ There is an isomorphism $f \colon (X_{\Gamma})_{K_M} \to (X_{G})_{K_M}$ satisfying $\pi_{G} \circ f =\pi_{\Gamma}.$ We observe that $\Q(h)$ is a function field for a model of $X_{\Gamma}$, isomorphic to $\mathbb{P}^{1}_{\Q}.$ Since we have models for $X_{\Gamma}$ and $X_G$ over $\Q$ there is a natural action of $\Gal(K_M/\Q)$ on $f.$ Using that $\pi_G$ is defined over $\Q$ we see that $f$ satisfies the following condition
\begin{equation}\label{equ:coc}
    \sigma(\pi_{\Gamma})=\pi_{\Gamma} \circ f^{-1}\circ \sigma(f)
\end{equation}
for every $\sigma \in \Gal(K_M/\Q).$

\cref{equ:coc} gives us a condition on the image of the cocycle \[\zeta \colon \Gal(K_M/\Q) \to \Aut((X_{\Gamma})_{K_M})=\PGL_2(K_M)\] that sends $\sigma$ to $f^{-1}\circ \sigma(f).$ 

For a fixed $M$ divisible by $N$, there are only finitely many cocycles from $\Gal(K_M/\Q)$ to $\PGL_2(K_M)$ that satisfy \cref{equ:coc}. We compute the cocycles on a set of generators of $\Gal(K_M/\Q)$ and extend to the whole group using the cocycle property. Let $\zeta$ be one such cocycle.  From \cref{lemma:coboundary P^1} and \cref{rmk:application P^1}, we know that if $\zeta$ is a coboundary in $H^{1}(\Gal(K_{M}/\mathbb{Q}),\PGL_2(K_M))$, then $\zeta$ gives a pair $(\PP^1_{\Q},\pi_1)$ which is isomorphic to a modular curve $(X_G,\pi_G.)$  

For congruence subgroups $\Gamma$ with label $7A^0$, $7C^0$, $7G^0$, $8A^0$, $8M^0$, $11A^0$, $14A^0$, $15A^0$, $16A^0$, $18A^0$ and $21A^0$; we verify using \texttt{Magma} that there exists a $\sigma \in \Gal(K_N/\Q)$ for which there is no degree $1$ function $g \in \Qab(t)$ satisfying the equation $\sigma(\pi_{\Gamma})=\pi_{\Gamma} \circ g.$ Since $\pi_{\Gamma} \in K_N(t)$ the action of $\Gal(K_M/\Q)$ for any multiple $M$ of $N$ factors through $\Gal(K_N/\Q).$ Therefore, for any mutiple $M$ of $N$ there exists a $\sigma \in \Gal(K_M/\Q)$ such that there is no function $g$ satisfying the equation $\sigma(\pi_{\Gamma})=\pi_{\Gamma} \circ g$.    Thus, there exists no pair ($X_{G},\pi_{G}$) such that $G \intersect \SL_2(\Zhat)$ is equal to $\Gamma.$ In these cases, \cref{THM:FIRST} is vacuously true.

For congruence subgroup $\Gamma$ with label $5F^0$, we find a pair $(\PP^1_{\Q},\pi_1)$ when $M=15.$ 

For congruence subgroups with labels $8E^0$ and $8K^0$, we find a pair $(\PP^1_{\Q},\pi_1)$ when $M=16.$ 

For congruence subgroup $\Gamma$ with label $16F^0$, we find a pair $(\PP^1_{\Q},\pi_1)$ when $M=32.$ 

For all the other congruence subgroups $\Gamma$, we find a pair $(\PP^1_{\Q},\pi_1)$ when $M=N$ where $N$ is the level of $\Gamma.$

\subsection{Breaking up the cocycles}

Let us denote the group $\Aut_{\Qab}(X_{G_1},\pi_{G_1})$ obtained in \cref{computing a pair} as $A_{G_1}.$ We would like to mention at this point that for every pair $(X_{G_1},\pi_{G_1})$ which we compute we have that $\Aut_{\Qab}(X_{G_1},\pi_{G_1})$ is defined over $K_N$, i.e., $A_{G_1}$ equals $\Aut_{K_N}(X_{G_1},\pi_{G_1})$, where $N$ is the level of corresponding congruence subgroup $\Gamma.$

Let $K:=K_{N^{\infty}}$ be the compositum of all fields $K_{N^e}$ with $e \ge 1.$ Let $K'$ be the compositum of all $K_m$ with $m$ relatively prime to $N.$ Then, $KK'=\Qab$ and $K \cap K' = \Q$. 

Let $\phi \colon \Gal(\Qab/\Q) \to A_{G_1}$ be a cocycle. Since, $A_{G_1}$ is defined over $K_N$ the group $\Gal(\Qab/\Q)$ acts via $\Gal(K/\Q).$

The following lemma shows that there are finitely many cocycles from $\Gal(K/\Q)$ to $A_{G_1}.$

\begin{lemma}\label{vertical cocycles}

Let $b$ be the least common multiple of orders of elements $a \in A_{G_1}$ that divide some power of $N$. If $N$ is congruent to $2$ modulo $4$, then any cocycle $\zeta \colon \Gal(K/\Q) \to A_{G_1}$ factors through $\Gal(K_{2Nb}/\Q).$ If $N$ is not congruent to $2$ modulo $4$, then any cocycle $\zeta \colon \Gal(K/\Q) \to A_{G_1}$ factors through $\Gal(K_{Nb}/\Q).$ 

\end{lemma}

\begin{proof}

Any cocycle $\zeta \colon \Gal(K/\Q) \to A_{G_1}$ factors through $\Gal(K_M/\Q)$ where $M$ is a multiple of $N$ with the same prime divisors as $N$. Define $G:=\Gal(K_M/K_N)$ if $N$ is not congruent to $2$ modulo $4$ and $G:=\Gal(K_M/K_{2N})$ if $N$ is congruent to $2$ modulo $4.$ The Galois group $G$ is cyclic and its cardinality has the same prime factors as $N$. Let $g$ be a generator of $G.$ Let $a:=\zeta(g).$ Then by our choice of $b$ we get that $a^b=1.$ Also observe that $\zeta$ is a homomorphism on $G$, since $A_{G_1}$ is defined over $K_N.$ Therefore, $\zeta(g^b)=1.$ It follows that $\zeta$ factors through $\Gal(K_{Nb}/\Q)$ if $N$ is not congruent to $2$ modulo $4$ and $\zeta$ factors through $\Gal(K_{2Nb}/\Q)$ if $N$ is congruent to $2$ modulo $4.$  
\end{proof}
Let $\{\zeta_i\}_{i=1}^{l}$ be the finite set of cocycles up to coboundaries from $\Gal(K/\Q)$ to $A_{G_1}.$ They can be computed using \cref{vertical cocycles}.
Fix a cocycle \[\zeta \colon \Gal(K/\Q) \to A_{G_1}.\] Fix $(X_{G'},\pi_{G'})$ as a $K$-twist of $(X_{G_1},\pi_{G_1})$ with respect to $\zeta$. There is an isomorphism $f \colon (X_{G_1})_{K} \to (X_{G'})_{K}$ such that $\zeta(\sigma)=f^{-1} \circ \sigma(f)$ is a representative of $\theta((X_{G'},\pi_{G'})).$  
 
\begin{lemma}\label{twisting two times}
Given a cocycle $\phi \colon \Gal(\Qab/\Q) \to A_{G_1}$ satisfying $\phi(i(\sigma,1))=\zeta(\sigma)$ for every $\sigma \in \Gal(K/\Q)$, there exists a cocycle $\gamma' \colon \Gal(K'/\Q) \to \Aut(X_{G'},\pi_{G'}),$ where $\Gal(K'/\Q)$ acts trivially on $\Aut(X_{G'},\pi_{G'})$ satisfying $\gamma'(\tau)=f \circ \phi(i(1,\tau)) \circ f^{-1}$ for every $\tau \in \Gal(K'/\Q).$ Conversely, given a cocycle $\gamma' \colon \Gal(K'/Q) \to \Aut(X_{G'},\pi_{G'})$, there exists a cocycle $\phi \colon \Gal(\Qab/\Q) \to A_{G_1}$ satisfying $\phi(i(\sigma,1))=\zeta(\sigma)$ for every $\sigma \in \Gal(K/\Q)$ and $\phi(i(1,\tau))=f^{-1} \circ \gamma'(\tau) \circ f$ for every $\tau \in \Gal(K'/\Q).$ 
\end{lemma}

\begin{proof}

Let $\phi \colon \Gal(\Qab/\Q) \to A_{G_1}$ be a cocycle satisfying $\phi(i(\sigma,1))=\zeta(\sigma)$ for every $\sigma \in \Gal(K/\Q).$ Let $\gamma \colon \Gal(K'/\Q) \to A_{G_1}$ be defined as $\gamma(\tau)=\phi(i(1,\tau)).$ Using \cref{Automorph over Q}, $\gamma(\tau)=f^{-1} \circ h_{\tau} \circ f$ for some $h_{\tau} \in \Aut(X_{G'},\pi_{G'}).$ Let $\gamma' \colon \Gal(K'/\Q) \to \Aut(X_{G'},\pi_{G'})$ defined as $\gamma'(\tau)=h_{\tau}.$ Since, $\gamma$ is a homomorphism, $\gamma'$ is also a homomorphism.

We will now show the converse. Define $\gamma \colon \Gal(K'/\Q) \to A_{G_1}$ as $\gamma(\tau)=f^{-1} \circ \gamma'(\tau) \circ f.$ Define \[\phi \colon \Gal(\Qab/\Q) \to A_{G_1}\] as $\phi(i(\sigma,\tau))= \gamma(\tau)\zeta(\sigma).$ Using \cref{breakup of cocs}, it follows that $\phi$ is a cocycle.
\end{proof}

\begin{lemma}\label{twisting P^1 to get a P^1}
Assume that $\zeta$ is a coboundary in $H^{1}(\Gal(K/\Q),\PGL_2(K)).$ A twist of $(X_{G_1},\pi_{G_1})$ by a cocycle $\phi \colon \Gal(\Qab/\Q) \to A_{G_1}$ satisfying $\phi(i(\sigma,1))=\zeta(\sigma)$ for every $\sigma \in \Gal(K/\Q)$ is isomorphic to $\PP^1_{\Q}$ if and only if twist of $(X_{G'},\pi_{G'})$ by the corresponding cocycle $\gamma'$ is isomorphic to $\PP^1_{\Q}.$
\end{lemma}

\begin{proof}
 Since $\zeta$ is a coboundary in $H^{1}(\Gal(K/\Q),\PGL_2(K))$, $\zeta(\sigma)=A^{-1}\sigma(A)$ for every $\sigma$ in $\Gal(K/\Q)$ and for some $A \in \PGL_2(K).$ Explicitly, $\gamma(\tau)=A^{-1}\gamma^{'}(\tau) A$, where $\gamma^{'}(\tau) \in \Aut(X_{G'},\pi_{G'}) \subsetneq \PGL_2(\Q).$ The inclusion $\Aut(X_{G'},\pi_{G'}) \subsetneq \PGL_2(\Q)$ uses the assumption that $X_{G'}$ is isomorphic to $\PP^1_{\Q}.$
 The cocycle $\phi \colon \Gal(\Qab/\Q) \to A_{G_1}$ which is given by $\phi(\sigma,\tau)= \zeta(\sigma)\sigma(\gamma(\tau))$ becomes $A^{-1}\sigma(A)\sigma(A^{-1}\gamma^{'}(\tau)A)=A^{-1}\gamma^{'}(\tau)\sigma(A).$ The two cocycles $\phi$ and $\gamma'$ are cohomologous. Therefore, the lemma holds.  

\end{proof}

\begin{remark}\label{generalizing twisting P^1 lemma}
\cref{twisting P^1 to get a P^1} holds for any pair of fields $(K,K')$ as long as they satisfy the two conditions; $K \intersect K'=\Q$ and $KK'=\Qab.$
\end{remark}

\cref{twisting P^1 to get a P^1} shows that if $X_{G'}$ is isomorphic to $\PP^1_{\Q}$, then twists of $(X_{G_1},\pi_{G_1})$ by cocycles $\phi \colon \Gal(\Qab/\Q) \to A_{G_1}$ that satisfy $\phi(i(\sigma,1))=\zeta(\sigma)$ for every $\sigma \in \Gal(K/\Q)$ and are isomorphic to $\PP^1_{\Q}$ can be obtained via homomorphisms from \[\Gal(K'/\Q) \to \Aut(X_{G'},\pi_{G'}) \to \PGL_2(\Q)\] that are coboundaries.

Let us now assume that $\zeta$ is not a coboundary in $H^{1}(\Gal(K/\Q),\PGL_2(K)).$ We want to classify the cocycles $\gamma' \colon \Gal(K'/\Q) \to \Aut(X_{G'},\pi_{G'})$ that give us $\PP^1_{\Q}.$ Since $\Gal(K'/\Q)$ acts trivially on $\Aut(X_{G'},\pi_{G'})$, the image of $\gamma^{'}$ is an abelian subgroup of $\Aut(X_{G'},\pi_{G'}).$ Let us denote it by $\calA.$ In all the cases where $\zeta$ is not a coboundary in $H^{1}(\Gal(K/\Q),\PGL_2(K))$ we find by computation that $\calA$ is either of order $1$, order $2$, order $3$, cyclic of order $4$ or Klein-$4$ group.  

In the following two lemmas we describe the cases where there exists no cocycle from $\Gal(K'/\Q)$ to $\calA$ that gives rise to a $\PP^1_{\Q}.$   

\begin{lemma}\label{obstruction 1}

If $\calA$ is a cyclic group of order 3, then the twist of $(X_{G'},\pi_{G'})$ by any cocycle $\gamma' \colon \Gal(K'/\Q) \to \calA$ is not isomorphic to $\PP^1_{\Q}$.
\end{lemma}

\begin{proof}

Assume that there exists a cocycle $\gamma' \colon \Gal(K'/\Q) \to \calA$ such that twist of $(X_{G'},\pi_{G'})$ by $\gamma'$ is isomorphic to $\PP^1_{\Q}$. We are given that $\calA$ is of order $3$. Therefore, $\gamma'$ factors through $\Gal(L/\Q)$, where $L$ is a degree $3$ Galois extension of $\Q.$ This would imply that $X_{G'}$ becomes isomorphic to $\PP^1_{\Q}$ over an extension of degree $3$ over $\Q$ which is a contradiction. 
\end{proof}
 
\begin{lemma}\label{obstruction 2}

Let $X_{G'}$ be the genus $0$ curve explicitly given as a conic $C:X^2+Y^2+Z^2=0 \subseteq \PP^2_{\Q}.$ Let $Q$ denote a Klein-4 twist of $C$  described by $[X:Y:Z] \to [-X:-Y:Z]$ and $[X:Y:Z] \to [X:-Y:-Z]$ over $K'$, here $K'= \bigcup_{(m,2)=1} K_m.$ Then, $Q$ has no rational points.
\end{lemma}

\begin{proof}

A Klein-4 twist $Q$ of $C$ is given by $dX^2+Y^2+eZ^2=0$ for squarefree integers $d,e.$ Further, since we are twisting over $K'$ we can assume that $d$ and $e$ are congruent to $1$ modulo $4$ and hence are squares in $\Q_2.$ It is now easy to verify that $Q$ has no $2$-adic points. Hence, the lemma holds. 
\end{proof}

In the cases where we get a curve $X_{G'}$ which is not isomorphic to $\PP^1_{\Q}$, we either face the obstructions discussed in \cref{obstruction 1} and \cref{obstruction 2} or we find a cocycle $\gamma' \colon \Gal(K'/\Q) \to \calA$ such that twist of $(X_{G'},\pi_{G'})$ by $\gamma'$ is $\PP^1_{\Q}$; it factors through $\Gal(L_0/\Q)$, where $L_0 \subseteq K_n$ and $n$ is $3$,$5$ or $15.$ Let \{$\gamma'_j$\} be the finite set of cocycles from $\Gal(K_n/\Q)$ to $ \calA$ such that twist of $(X_{G'},\pi_{G'})$ by each $\gamma'_j$ is $\PP^1_{\Q}.$ We compute all these twists $(\PP^1_{\Q},\pi_{\gamma'_j}).$ 

We now discuss how can we get all the other cocycles $\gamma' \colon \Gal(K'/\Q) \to \calA$ such that twist of $(X_{G'},\pi_{G'})$ by $\gamma'$ is $\PP^1_{\Q}.$ Suppose there is another cocycle $\gamma'' \colon \Gal(K'/\Q) \to \calA$ that gives us a $\PP^1_{\Q}$ and factors through $\Gal(L_1/\Q)$, where $L_1$ is a finite Galois extension of $\Q.$ Since we have computed the pairs $(\PP^1_{\Q},\pi_{\gamma'_j})$  corresponding to cocycles $\phi \colon \Gal(KK_n/\Q) \to A_{G_1} \to \PGL_2(KK_n)$ which satisfy the condition $\phi(i(\sigma,1))=\zeta(\sigma)$ for every $\sigma \in \Gal(K/\Q)$, we can assume that $L_1 \subseteq K_m$, where $m$ is coprime to $n.$ Now to obtain all the other coboundaries $\phi \colon \Gal(\Qab/\Q) \to A_{G_1} \to \PGL_2(\Qab)$ satisfying $\phi(i(\sigma,1))=\zeta(\sigma)$ for every $\sigma \in \Gal(K/\Q)$ we apply \cref{twisting P^1 to get a P^1} to every pair $(\PP^1_{\Q},\pi_{\gamma'_j}).$

\subsection{Concluding the proof}

Let $\{(X_{\frakG_i},\pi_{\frakG_i})\}_{i=2}^s$ be the finite set of genus $0$ modular curves isomorphic to $\PP^1_{\Q}$ such that either $(X_{\frakG_i},\pi_{\frakG_i})$ is a twist of $(X_{G_1},\pi_{G_1})$ with respect to a coboundary $\zeta \colon \Gal(K/\Q) \to \PGL_2(K)$ or $(X_{\frakG_i},\pi_{\frakG_i})$ is a twist of $(X_{G_1},\pi_{G_1})$ with respect to a coboundary $\phi \colon \Gal(KK_n/\Q) \to \PGL_2(KK_n)$ satisfying $\phi(i(\sigma,1))=\zeta(\sigma)$ for every $\sigma \in \Gal(K/\Q)$ and $\zeta \colon \Gal(K/\Q) \to \PGL_2(K)$ is not a coboundary.   Further, if $(X_G,\pi_G)$ is any other $\Qab$-twist of $(X_{G_1},\pi_{G_1})$ isomorphic to $\PP^1_{\Q}$, then using \cref{twisting P^1 to get a P^1} and \cref{generalizing twisting P^1 lemma} $(X_G,\pi_G)$ can be obtained as a $\Qab$-twist of some $(X_{\frakG_i},\pi_{\frakG_i})$ for $i \in \{1,\ldots,l\}$ by considering cocycles $\gamma \colon \Gal(\Qab/\Q) \to \Aut(X_{\frakG_i},\pi_{\frakG_i}).$

Hence, \cref{THM:FIRST} holds.

\section{Proof of \cref{THM:SECOND}}\label{proof 2}

From \cref{THM:FIRST} we know that any modular curve $X_G$ isomorphic to $\PP^1_{\Q}$ is a $\Qab$-twist of $(\PP^1_{\Q},\pi_i)$ for some $i \in \{1,\ldots,r\}$ via a homomorphism $\gamma \colon \Gal(\Qab/\Q) \to \calA_i \subseteq \Aut(\PP^1_{\Q},\pi_i).$ 

Let $(X_G,\pi_G)$ be a modular curve isomorphic to $(\PP^1_{\Q},\pi_i)$ for some $i \in \{1,\ldots,r\}$.
Let $\gamma \colon \Gal(\Qab/\Q) \to \Aut(X_G,\pi_G)$ be a homomorphism. Since $\Gal(\Qab/\Q)$ is abelian, the image of $\gamma$ denoted by $\calA$ is a finite abelian group. We compute all such possible finite abelian subgroups up to conjugation in $\PGL_2(\Q).$ We list these possibilities in \cref{sec:families of twists}. We are interested in computing the cocycles $\gamma \colon \Gal(\Qab/\Q) \to \Aut(X_G,\pi_G)$ such that the twist of $(X_G,\pi_G)$ by $\gamma$ is isomorphic to $\PP^1_{\Q}.$ This is described in \cref{sec:families of twists} for various cases of $\calA$ that occur. Furthermore, we can compute the morphism $\pi'$ to j-line of the twist as follows. We factor $\pi_G$ as $J_i \circ u_i$, where $u_i \colon X_G \to \PP^1_{\Q}$ is the function $\pi$ described in the appropriate case of $\calA$ we are in. Then define $\pi':=J_i \circ \pi_{\calA,v}$, where $\pi_{\calA,v}(u) \in \Q(u)$ is given explicitly in each case.

\section{Computing generators for $G$ and finding equation of $X_G$ as a conic using generators of $G$}\label{sec:Applications}

Fix a genus $0$ congruence subgroup $\Gamma$ of level $N$ containing $-I$. Let $h$ be the normalized  hauptmodul of $\Gamma$, see \cref{sec:Haupt} for computation of hauptmoduls. Fix a genus 0 modular curve $(X_G,\pi_G)$ such that $X_G$ is isomorphic to $\PP^1_{\Q}$ and $G \intersect \SL_2(\ZZ)=\Gamma.$ Let $C$ be a matrix in $\PGL_2(\Qab)$ such that $\pi_G(C(h))=j$. Let $M$ be the smallest positive integer satisfying $N|M$ such that $C \in \PGL_2(K_M).$ Let $H := \pi_M(\Gamma).$ In this section we will describe a procedure to compute generators for $G$. Let us denote $\Aut_{K_M}(X_G,\pi_G)$ by $A_G.$ Observe that $A_G$ is a finite subgroup of $\Aut_{K_M}(X_G) \subsetneq PGL_{2}(K_{M})$.

Let us denote $C(h)$ as $u$ for this section. Consider the set $\Delta=\{\xi(u)|~\xi \in A_G\}.$

\begin{lemma}

For $A$ in $N_{GL_{2}(\ZZ/M\ZZ)}(H)/(H)$ and $\xi(u)$ in $\Delta$, define $\xi(u).A=\sigma_{d}(\xi)(u*A)$ where $d=\det(A).$ This is an action of $N_{GL_{2}(\ZZ/M\ZZ)}(H)/(H)$ on $\Delta$. 
\end{lemma}

\begin{proof}

Let $A \in$ $N_{GL_{2}(\ZZ/M\ZZ)}(H)$. For every $h_1$ in $H$ we have $(u*A)*h_1=u*(Ah_1)=u*(h_2A)=(u*h_2)*A=u*A$ where $h_2 \in H.$ Therefore, $u*A$ lies in $K_M(h).$ Since, $j*A=j$, $(\pi_G(u))*A=\pi_G(u*A)=j$. Therefore, $u*A=\phi(u)$ where $\phi \in \PGL_2(K_M).$ Since, $\pi_G(u)=\pi_G(\phi(u))=j$ we get $\pi_G(\phi(C))=\pi_G(C)$, therefore, $\pi_G(\phi)=\pi_G$. Hence, $\phi(u)$ lies in $\Delta$.

We know $(\xi(u)).A=\sigma_{d}(\xi)(u*A)$. Since, $\pi_G(\sigma_{d}(\xi))=\pi_G$, $(\xi(u)).A$ lies in $\Delta$. 

Since, $H$ acts trivially we get an action of the quotient group.
\end{proof}

Let $\Tilde{G}$ be the stabilizer of $\{u\}$ under this action. Taking inverse image of $\Tilde{G}$ under quotient map (from $N_{GL_{2}(\ZZ/M\ZZ)}(H)$ to $N_{GL_{2}(\ZZ/M\ZZ)}(H)/(H)$) gives us our required subgroup $G$.

\subsection{Alternate method to compute the generators for $G$}

We compute the following set $$S:=\{B \subseteq N_{GL_{2}(\ZZ/M\ZZ)}(H)/(H)~|~\det(B) \simeq (\ZZ/M\ZZ)^{\times}\}.$$ After we compute $S$, we search for the element which fixes $u$, taking its inverse image of under quotient map gives us our required subgroup $G$.  

\subsection{Computing equation for the curve $X_G$ from $G$}

Fix an open subgroup $G$ of $\GL_2(\Zhat)$ such that $\det(G)=\Zhat^{*}$, $-I \in G$ and $X_G$ has genus $0$. Let $N$ be the level of $G$. Fix a set $S$ of generators for $\pi_N(G)$ in $\GL_2(\ZZ/N\ZZ).$ In this section we will explain how to compute an explicit model for $X_G.$

Let $\Gamma$ be the genus $0$ congruence subgroup $G \intersect \SL_2(\ZZ)$ and let $H:=\pi_N(G) \intersect \SL_2(\ZZ/N\ZZ).$ Let $h$ be the normalized hauptmodul of $\Gamma$, see \cref{sec:Haupt} for computation of hauptmoduls. The quotient group $G/H$ is isomorphic to the Galois group $\Gal(K_N/\Q)$ via the determinant function. Using the definition of $H$ we have that $G \subseteq N_{\GL_2(\ZZ/N\ZZ}(H).$ Therefore, $h*g \in K_N(h)$ for every $g \in G$ where $*$ is the right action of $\GL_2(\ZZ/N\ZZ)$ on $\calF_N$ as described in section 2. For $g \in G$, let $d_g=\det(g).$ From the equation $\pi_{\Gamma}(h)=j$ we get that $\sigma_{d_g}(\pi_{\Gamma})(h*g)=j$ for every $g \in G.$ Thus, $h*g=\phi_{\sigma_{d_g}}(h)$ where $\phi_{\sigma_{d_g}} \in \PGL_2(K_N).$ Since $H$ acts trivially on $h$ we get that the quotient group $G/H$ acts on $h$.

\begin{lemma}
Let $\zeta \colon \Gal(K_N/\Q) \to \PGL_2(K_N)$ be the map defined as $\zeta(\sigma_g)=\phi_{\sigma_{d_g}}^{-1}.$ The map $\zeta$ is a cocycle.
\end{lemma}

\begin{proof}
We know that $*$ is a right action. Therefore, $(h*g_1)*g_2=h*(g_1g_2).$ This gives us $\sigma_{d_{g_2}}(\phi_{\sigma_{d_{g_1}}})\phi_{\sigma_{d_{g_2}}}(h)=\phi_{\sigma_{d_{g_1g_2}}}(h).$ Thus, $\zeta(\sigma_{g_2}\sigma_{g_1})=\zeta(\sigma_{g_1}\sigma_{g_2})=\phi_{\sigma_{d_{g_1g_2}}}^{-1}=\phi_{\sigma_{d_{g_2}}}^{-1}\sigma_{d_{g_2}}(\phi_{\sigma_{d_{g_1}}}^{-1})=\zeta(\sigma_{g_2})\sigma_{d_{g_2}}(\zeta(\sigma_{g_1})).$
\end{proof}

Using the procedure described in \cref{sec:Twists} we can now get an equation for the modular curve $X_G$ as a conic in $\PP^2.$ 

\appendix
\section{How to read the tables}\label{sec: Tables}

In tables \ref{tab: Table 1A-1A-6C-6A} to \ref{tab: Table 30A-120D-48A-48H} we use a label of the form $NA_1-MA_2$ to denote the groups $G_i$ as in \cref{THM:FIRST}. Here $N$, $M$ are natural numbers and $A_1$, $A_2$ are letters which need not be equal. The natural number $M$ is the level of $G$. The part $NA$ denotes the label of congruence subgroup $\Gamma$ equal to $G_i \intersect \SL_2(\Zhat)$ from \cite{MR2016709}. The label 1A-1A denotes the $j$-line. The function given in the column "$\pi(t)$" describes the morphism $X_G \to X_{G^{'}}$ where $G'$ is a supergroup containing $G$ whose label is given in the column "sup". Thus, to obtain the map $\pi_G \colon X_G \to \PP^1_\Q$ we can compose the functions $X_G \to X_{G^{'}} \to \PP^1_Q.$ In few cases we describe $\pi(t)$ by a sequence of the form $[f_1,f_2,\ldots,f_r]$ where each $f_i \in \Q(t)$ and the corresponding $\pi(t)$ is $f_r(\ldots(f_2(f_1(t)))).$

To determine the family described by $(X_G,\pi_G)$ we use tables \ref{tab: Table fam 2A-2A-6E-24B} to \ref{tab: Table fam 28A-56D-48A-48H}. We can find a matrix $C$ (easily computable from MAGMA) such that a maximal abelian subgroup of $\Aut(\PP^1_{\Q},\pi(C(t)))$ is equal to  $\calA$ of the case mentioned in column "Case." The column "Case" refers to the case number in \cref{sec:families of twists}.

Let us go through an example. 

\begin{example}\label{cubic family}

Define \[\pi(t):= \frac{(1728t^6 - 5184t^5 + 10368t^4 - 12096t^3 + 10368t^2 - 5184t + 1728)}{(t^6 -3t^5 - 3t^4/4 + 13t^3/2 - 3t^2/4 - 3t + 1)}.\] The pair $(\PP^1_{\Q},\pi)$ is labeled $2C-2A$ in \cref{tab: Table 1A-1A-6C-6A}. 

Let $\calA_1 =\{(\begin{smallmatrix}
1 & 0 \\ 
0  & 1
\end{smallmatrix}),(\begin{smallmatrix}
-1 & 0 \\ 
0  & 1
\end{smallmatrix})\}\isom \ZZ/2\ZZ$; it is $\calA$ given in Case 1 of \cref{sec:families of twists} and $\calA_2 =\{(\begin{smallmatrix}
1 & 0 \\ 
0  & 1
\end{smallmatrix}),(\begin{smallmatrix}
0 & -1 \\ 
1  & -1
\end{smallmatrix}),(\begin{smallmatrix}
1 & -1 \\ 
1  & 0
\end{smallmatrix})\}\isom \ZZ/3\ZZ$; it is $\calA$ given in Case 3 of \cref{sec:families of twists}.

Define the rational functions
\[J_1(t):=\frac{(t^3 + 576t^2 + 110592t + 7077888)}{(t^2 - 128t + 4096)}\] and
\[J_2(t):=\frac{(1728t^2 - 5184t + 15552)}{(t^2 - 3t + 9/4)}.\]

Then, any genus $0$ modular curve $(X_G,\pi_G)$ such that $X_G$ has a rational point and is a $\Qab$-twist of $(\PP^1_{\Q},\pi)$ is isomorphic to $(\PP^1_{\Q},J_i \circ \pi_{\calA_i,v})$ for some $i \in \{1,2\}.$
 
\end{example}
\begin{remark}
\begin{enumerate}
    \item The tables are automated. The maps can be chosen to make the columns look nicer. We can also include more information here if needed.

    \item We illustrate the similarity of our labels with those given in \cite{MR3671434} with an example. The pair $(X_{G^t},\pi)$ with label $4E-4A$ in table \ref{tab: Table 1A-1A-6C-6A} is exactly isomorphic to one of the modular curves with lables $4E^0-4a$, $4E^0-4b$ and $4E^0-4c$ in \cite{MR3671434}, i.e., our labels match with labels in \cite{MR3671434} up to first three symbols. The generators for groups $G$ can be accessed in the file "Tables 1-9.m" of github repository {\url{https://github.com/Rakvi6893/Classification-of-genus-0-Modular-Curves-that-have-a-Rational-Point}}.
    
\end{enumerate}

\end{remark}

\newpage

\begin{table}[H] 
\begin{small}
\begin{tabular}{|p{1.5 cm}|p{14 cm}|p{1.5 cm}|}\hline 

{Label} & {$\pi(t)$} & {Sup} \\ \hline 
1A-1A   & {} & {}  \\ \hline
2A-2A & \tiny{$t^2 + 1728$} & {1A-1A}  \\ \hline 
2B-2A & \tiny{$(t^3 + 24t^2 + 192t + 512)/(t + 24)$} & {1A-1A}  \\ \hline 
2C-2A & \tiny{$(t^2 - 8t + 256)/(t - 8)$} & {2B-2A}  \\ \hline 

3A-3A & \tiny{$t^3$} & {1A-1A}  \\ \hline 
3B-3A & \tiny{$(3645t^4/4 + 243t^3/2 + 27t^2/8 - t/8 -
1/192)/(t^4 + t^3/24)$} & {1A-1A}  \\ \hline 
3C-3A & \tiny{$(-729t^6/8 - 243t^5/8 - 27t^4/32 + 7t^3/16 +
t^2/128 - t/384 + 1/13824)/(t^6 - t^5/4 + t^4/48 -
t^3/1728)$} & {1A-1A}  \\
\hline 

\end{tabular} 
\vspace{1 cm }
\caption{Table of $\pi$ for groups with label $1A-1A$ to $6C-6A$}

\label {tab: Table 1A-1A-6C-6A}
\end{small}
\end{table}

\newpage
\begin{table}[H] 
\begin{small}
\begin{tabular}{|p{1.5 cm}|p{14 cm}|p{1.5 cm}|}\hline 

{Label} & {$\pi(t)$} & {Sup} \\ \hline

6D-6A & \tiny{$t^3 - 24$} & {2B-2A} \\ \hline
6E-6A & \tiny{$t^2/12/(t^2 + 1/72)$} & {3C-3A} \\ \hline
6E-12A & \tiny{$1/(2t^2 + 6)$} & {3C-3A} \\ \hline
6E-24A & \tiny{$t^2/(12t^2 + 1/36)$} & {3C-3A} \\ \hline
6E-24B & \tiny{$t^2/(12t^2 + 1/3)$} & {3C-3A} \\ \hline
6F-6A & \tiny{$(3t^3/88 - t^2/176)/(t^3 - 3t^2/22 - 3t/176 +
1/352)$} & {3B-3A} \\ \hline
6G-6A & \tiny{$(3t^3/44 - t^2/88)/(t^3 - 3t^2/22 - 3t/88 +
1/176)$} & {3C-3A} \\
\hline

\end{tabular} 
\vspace{0.5 cm }
\caption{Table of $\pi$ for groups with label $6D-6A$ to $8F-8A$}

\label {tab: Table 6D-6A-8F-8A}
\end{small}
\end{table}

\newpage
\begin{table}[H] 
\begin{small}
\begin{tabular}{|p{1.5 cm}|p{14 cm}|p{1.5 cm}|}\hline 

{Label} & {$\pi(t)$} & {Sup} \\ \hline

8G-8A & \tiny{$[
        (t^2 + 2t)/(t - 2),
        t/(t^2 + 4),
        t^2/(t - 1/8),
        (54000t^3 + 5400t^2 + 180t + 2)/(t^3 - t^2 + t/4)
    ]$} & {1A-1A} \\ \hline

8H-8A & \tiny{$(t^2 + 2)/t$} & {8D-8A} \\ \hline
8I-8A & \tiny{$(-24t^8 + 832t^6 + 448t^4 - 768t^2 - 384)/(t^8 +
8t^6 + 24t^4 + 32t^2 + 16)$} & {2B-2A} \\ \hline
8J-8A & \tiny{$(t^4 + 8t^2 + 32)/(t^2 + 4)$} & {2C-2A} \\ \hline
8K-16A & \tiny{$8t^4/(t^4 + 4t^2 + 2)$} & {4C-4A} \\ \hline
8L-8A & \tiny{$(-24t^8 - 192t^7 - 480t^6 - 192t^5 - 3472t^4 +
192t^3 - 480t^2 + 192t - 24)/(t^8 + 8t^7 + 20t^6 + 8t^5 -
26t^4 - 8t^3 + 20t^2 - 8t + 1)$} & {2B-2A} \\ \hline
8N-8A & \tiny{$(t^2 - 1)/t$} & {4G-4A} \\
\hline

\end{tabular} 
\vspace{0.5 cm }
\caption{Table of $\pi$ for groups with label $8G-8A$ to $8N-120C$}

\label {tab: Table 8G-8A-8N-120C}
\end{small}
\end{table}

\newpage
\begin{table}[H] 
\begin{small}
\begin{tabular}{|p{1.5 cm}|p{14 cm}|p{1.5 cm}|}\hline 

{Label} & {$\pi(t)$} & {Sup} \\ \hline
8N-160A & \tiny{$(584t^8/161 - 64t^7 - 2336t^6/23 + 448t^5 +
5840t^4/23 - 448t^3 - 2336t^2/23 + 64t + 584/161)/(t^8 +
584t^7/161 - 28t^6 - 584t^5/23 + 70t^4 + 584t^3/23 - 28t^2 -
584t/161 + 1)$} & {4B-4A} \\ \hline

8N-160B & \tiny{$(-1288t^8/73 - 64t^7 + 36064t^6/73 + 448t^5 -
90160t^4/73 - 448t^3 + 36064t^2/73 + 64t - 1288/73)/(t^8 -
1288t^7/73 - 28t^6 + 9016t^5/73 + 70t^4 - 9016t^3/73 - 28t^2 +
1288t/73 + 1)$} & {4B-4A} \\ \hline

8O-8A & \tiny{$(-t^8 - 20t^6 + 26t^4 - 20t^2 - 1)/(t^7 + t^5 -
t^3 - t)$} & {2C-2A} \\ \hline
8P-8A & \tiny{$(-3t^4/16 - t^3/2 - t^2/8 + 1/16)/(t^4 + t^3 +
t^2 + t)$} & {8B-8C} \\ \hline
8P-8B & \tiny{$(t^3/2 - t^2/2 + t/4)/(t^4 - 2t^3 + t^2 - t
+ 1/4)$} & {8B-8C} \\ \hline

9A-9A & \tiny{$(-6t^3 - 9t^2/2 - 1/8)/t^3$} & {3A-3A}  \\ \hline 
9B-9A & \tiny{$(t^{12} - 72t^9 + 1728t^6 - 13824t^3)/(t^3 - 27)$} &
{1A-1A} \\
\hline

\end{tabular} 
\vspace{0.5 cm }
\caption{Table of $\pi$ for groups with label $8N-160A$ to $12C-12C$}

\label {tab: Table 8N-160A-12C-12C}
\end{small}
\end{table}

\newpage
\begin{table}[H] 
\begin{small}
\begin{tabular}{|p{1.5 cm}|p{14 cm}|p{1.5 cm}|}\hline 
{Label} & {$\pi(t)$} & {Sup} \\ \hline 
12C-12D & \tiny{$(-24t^6 - 1/64)/t^6$} & {2B-2A} \\ \hline
12C-24A & \tiny{$(-24t^6 - 1/216)/t^6$} & {2B-2A} \\ \hline
12C-24B & \tiny{$6t^2$} & {6D-6A} \\ \hline
12C-24C & \tiny{$2t^2$} & {6D-6A} \\ \hline
12C-24D & \tiny{$-512t^6 - 24$} & {2B-2A} \\ \hline
12D-12A & \tiny{$-1728t^6 + 1728t^4 - 576t^2 + 40$} & {2B-2A} \\ \hline
12E-12A & \tiny{$1/(24t^2 + 8)$} & {6F-6A} \\
\hline

\end{tabular} 
\vspace{1 cm }
\caption{Table of $\pi$ for groups with label $12C-12D$ to $16D-16A$}

\label {tab: Table 12C-12D-16D-16A}
\end{small}
\end{table}

\newpage
\begin{table}[H] 
\begin{small}
\begin{tabular}{|p{1.5 cm}|p{14 cm}|p{1.5 cm}|}\hline 

{Label} & {$\pi(t)$} & {Sup} \\ \hline

16E-16A & \tiny{$t^2 - 4$} & {8D-8A} \\ \hline
16F-32A & \tiny{$(-t^8/8 + t^6/4 - 3t^4/16 + t^2/16 -
1/128)/(t^8 + 12t^7 + 30t^6 + 42t^5 + 67t^4/2 + 21t^3 +
15t^2/2 + 3t/2 + 1/16)$} & {4A-4A} \\ \hline
16F-32B & \tiny{$(-t^8/552 - t^7/69 - 5t^6/138 - t^5/69 +
13t^4/276 + t^3/69 - 5t^2/138 + t/69 - 1/552)/(t^8 - 80t^7/23
+ 316t^6/69 - 96t^5/23 + 166t^4/69 - 16t^3/23 - 20t^2/69 -
1/23)$} & {4A-4A} \\ \hline
16G-16A & \tiny{$  [(-t^2 + 1/2)/(t^2 - 2t - 1/2),
        (t^2 + t)/(t - 1),
        t^2/(t - 1/6),
        (t^2 - 6t)/(t + 2/21),
        (97336000t^3/321489 + 1692800t^2/189 + 88320t + 290304)/(t^3 +
            358t^2/21 + 844t/9 + 168)
    ]$} & {1A-1A} \\ \hline
16G-16B & \tiny{$[(t^2 + 1)/t,(t^4 + 24t^2 + 16)/(t^3 + 4t)]$} & {2C-2A} \\ \hline
16G-16C & \tiny{$[(-t^2 + 1)/(2t),
        t/(t^2 + 1),
        t^2,
        t^2 - t/4,
        (16777216t^3 + 589824t^2 + 6912t + 27)/(t + 1/64)
    ]$} & {1A-1A} \\ \hline
16G-16D & \tiny{$[(t^2 - 1/2)/t,
        t/(t^2 + 2),
        t^2,
        t^2 - t/8,
        (268435456t^3 + 786432t^2 + 768t + 1/4)/t
    ]$} & {1A-1A} \\
\hline

\end{tabular} 
\vspace{1 cm }
\caption{Table of $\pi$ for groups with label $16E-16A$ to $16G-240B$}

\label {tab: Table 16E-16A-16G-240B}
\end{small}
\end{table}

\newpage
\begin{table}[H] 
\begin{small}
\begin{tabular}{|p{1.5 cm}|p{14 cm}|p{1.5 cm}|}\hline 
{Label} & {$\pi(t)$} & {Sup} \\ \hline

16G-240C & \tiny{$[
        (t^2 + 2t)/(t + 1/3),
        (t^2 - 3t)/(t - 4/3),
        (t^2 + t/3)/(t + 27),
        t/(t^2 + 1/9),
        (257924804t^3/32805 - 76862878t^2/3645 + 22905521t/1215 -
            13651919/2430)/(t^3 + 3t^2/2 - 9t/4 - 27/8)
    ]$} & {1A-1A} \\ \hline
16H-16A & \tiny{$(t^2 - 1/2)/t$} & {16C-16A} \\ \hline

18B-18A & \tiny{$t^2$} & {9C-9A} \\ \hline
18B-18B & \tiny{$(-5159780352t^{24} + 107495424t^{18} - 466560t^{12} +
756t^6 - 27/64)/t^6$} & {1A-1A} \\ \hline
18B-36A & \tiny{$1/(12t^2)$} & {9C-9A} \\ \hline
18B-36B & \tiny{$(-262144t^{24} + 147456t^{18} - 17280t^{12} + 756t^6 -
729/64)/t^6$} & {1A-1A} \\ \hline

18B-72A & \tiny{$(-157464t^{24} + 756t^{18} - 5t^{12}/4 + t^6/1296 -
1/10077696)/t^{18}$} & {1A-1A} \\
\hline
\end{tabular} 
\vspace{0.5 cm }
\caption{Table of $\pi$ for groups with label $16G-240C$ to $24B-24D$}

\label {tab: Table 16G-240C-24B-24D}
\end{small}
\end{table}

\newpage
\begin{table}[H] 
\begin{small}
\begin{tabular}{|p{1.5 cm}|p{14 cm}|p{1.5 cm}|}\hline 
{Label} & {$\pi(t)$} & {Sup} \\ \hline 
24B-24E & \tiny{$(16t^{48} - 144t^{44} + 1242t^{40} - 6021t^{36} +
407511t^{32}/16 - 583929t^{28}/8 + 10727235t^{24}/64 - 33126489t^{20}/128 +
977319999t^{16}/4096 - 521343621t^{12}/4096 + 1248354909t^8/32768 -
387420489t^4/65536 + 387420489/1048576)/(t^{40} - 7t^{36} + 135t^{32}/8 -
243t^{28}/16 + 729t^{24}/256)$} & {1A-1A} \\ \hline
24B-24F & \tiny{$(4738381338321616896t^{48} - 131621703842267136t^{44} +
1472619371692032t^{40} - 8541693886464t^{36} + 27799324416t^{32} -
52348032t^{28} + 58860t^{24} - 89t^{20}/2 + 559t^{16}/20736 -
223t^{12}/20155392 + 23t^8/5804752896 - t^4/1253826625536 +
1/6499837226778624)/(t^{24} - t^{20}/108 + 5t^{16}/279936 -
7t^{12}/544195584 + t^8/313456656384)$} & {1A-1A} \\ \hline
24B-24G & \tiny{$(4738381338321616896t^{48} + 131621703842267136t^{44} +
1472619371692032t^{40} + 8541693886464t^{36} + 27799324416t^{32} +
52348032t^{28} + 58860t^{24} + 89t^{20}/2 + 559t^{16}/20736 +
223t^{12}/20155392 + 23t^8/5804752896 + t^4/1253826625536 +
1/6499837226778624)/(t^{24} + t^{20}/108 + 5t^{16}/279936 +
7t^{12}/544195584 + t^8/313456656384)$} & {1A-1A} \\ \hline

24B-24H & \tiny{$(1156831381426176t^{48} + 128536820158464t^{44} +
5752419420672t^{40} + 133463966976t^{36} + 1737457776t^{32} + 13087008t^{28}
+ 58860t^{24} + 178t^{20} + 559t^{16}/1296 + 223t^{12}/314928 +
23t^8/22674816 + t^4/1224440064 + 1/1586874322944)/(t^{24} + t^{20}/27
+ 5t^{16}/17496 + 7t^{12}/8503056 + t^8/1224440064)$} & {1A-1A} \\ \hline
25A-25A & \tiny{$16t^5 + 20t^3 + 5t + 5/2$} & {5B-5A} \\ \hline
25B-25A & \tiny{$(-t^2/2 - t + 1/8)/(t^2 - t/2 - 1/4)$} & {25A-25A} \\ \hline
25B-25B & \tiny{$(5t^{10}/2 + 10t^9 - 25t^8/2 + 120t^7 + 25t^6 +
252t^5 - 25t^4 + 120t^3 + 25t^2/2 + 10t - 5/2)/(t^{10} - 5t^8
+ 10t^6 - 10t^4 + 5t^2 - 1)$} & {5B-5A} \\
\hline

\end{tabular} 
\vspace{1 cm }
\caption{Table of $\pi$ for groups with label $24B-24E$ to $30A-120C$}

\label {tab: Table 24B-24E-30A-120C}
\end{small}
\end{table}

\newpage
\begin{table}[H] 
\begin{small}
\begin{tabular}{|p{1.5 cm}|p{14 cm}|p{1.5 cm}|}\hline 
{Label} & {$\pi(t)$} & {Sup} \\ \hline

30A-120D & \tiny{$(-3t^6 + 1/16)/t^6$} & {5B-5A} \\ \hline

30A-120E & \tiny{$(-3t^6 + 1/27648)/t^6$} & {5B-5A} \\ \hline
30A-120F & \tiny{$(-3t^6 + 1/3456000)/t^6$} & {5B-5A} \\ \hline
30A-120G & \tiny{$(-3t^6 - 1/746496)/t^6$} & {5B-5A} \\ \hline

30A-120H & \tiny{$(-3t^6 - 1/93312000)/t^6$} & {5B-5A} \\ \hline
32A-32A & \tiny{$-1/t^2$} & {16C-16B} \\ \hline
32A-32B & \tiny{$-t^{16} + 16t^8 - 24$} & {2B-2A} \\ \hline
32A-32C & \tiny{$-1/(2t^2)$} & {16C-16B} \\ \hline
32A-32D & \tiny{$-256t^{16} + 256t^8 - 24$} & {2B-2A} \\ \hline
36A-36A & \tiny{$t^2$} & {18B-18A} \\ \hline
36A-36B & \tiny{$(46656t^{24} - t^{12}/96 - 1/5159780352)/t^{18}$} & {2A-2A} \\ \hline
36A-36C & \tiny{$-4t^4$} & {9C-9A} \\ \hline
36A-36D & \tiny{$(-101559956668416t^{48} + 78364164096t^{36} - 12597120t^{24}
+ 756t^{12} - 1/64)/t^{12}$} & {1A-1A} \\ \hline
36A-72A & \tiny{$(-10077696t^{24} - 3888t^{12} + 1/8)/t^6$} & {2A-2A} \\ \hline
36A-72B & \tiny{$1/(2t^2)$} & {18B-18A} \\ \hline
36A-72C & \tiny{$-1/(16t^4)$} & {9C-9A} \\ \hline
36A-72D & \tiny{$(-2176782336t^{48} + 756t^{36} - 5t^{24}/55296 +
t^{12}/247669456896 - 1/26623333280885243904)/t^{36}$} & {1A-1A} \\ \hline

48A-48A & \tiny{$t^2$} & {24A-24A} \\ \hline
48A-48B & \tiny{$6t^2$} & {24A-24G} \\ \hline
48A-48C & \tiny{$3t^2$} & {24A-24G} \\ \hline
48A-48D & \tiny{$12t^2$} & {24A-24A} \\ \hline
48A-48E & \tiny{$2t^2$} & {24A-24A} \\ \hline
48A-48F & \tiny{$6t^2$} & {24A-24A} \\ \hline
48A-48G & \tiny{$4t^2$} & {24A-24G} \\ \hline

48A-48H & \tiny{$1/(2t^2)$} & {24A-24G} \\ \hline

\end{tabular} 
\vspace{1 cm }
\caption{Table of $\pi$ for groups with label $30A-120D$ to $48A-48H$}

\label {tab: Table 30A-120D-48A-48H}
\end{small}
\end{table}

\newpage

\begin{table}[H] 
\begin{small}
 
\vspace{2 cm }
\caption{Table of families for groups with label $2A-2A$ to $6E-24B$}

\label {tab: Table fam 2A-2A-6E-24B}
\end{small}
\end{table}

\newpage

\begin{table}[H] 
\begin{small}
 
\vspace{1 cm }
\caption{Table of families for groups with label $6F-6A$ to $8J-8A$}

\label {tab: Table fam 6F-6A-8J-8A}
\end{small}
\end{table}

\newpage

\begin{table}[H] 
\begin{small}
 
\vspace{1 cm }
\caption{Table of families for groups with label $8K-16A$ to $9J-9A$}

\label {tab: Table fam 8K-16A-9J-9A}\end{small}

\end{table}

\newpage

\begin{table}[H] 
\begin{small}
 
\vspace{2 cm }
\caption{Table of families for groups with label $10A-10A$ to $13C-13A$}

\label {tab: Table fam 10A-10A-13C-13A}
\end{small}
\end{table}

\newpage

\begin{table}[H] 
\begin{small}
 
\vspace{2 cm }
\caption{Table of families for groups with label $14B-14A$ to $16H-16A$}

\label {tab: Table fam 14B-14A-16H-16A}
\end{small}
\end{table}

\newpage

\begin{table}[H] 
\begin{small}
 
\vspace{1 cm }
\caption{Table of families for groups with label $18B-18A$ to $28A-56C$}

\label {tab: Table fam 18B-18A-28A-56C}
\end{small}
\end{table}

\newpage

\begin{table}[H] 
\begin{small}
 
\vspace{2 cm }
\caption{Table of families for groups with label $28A-56D$ to $48A-48H$}

\label {tab: Table fam 28A-56D-48A-48H}
\end{small}
\end{table}


\begin{bibdiv} 
\begin{biblist}

\bib{MR2681719}{article}{
   author={Beauville, Arnaud},
   title={Finite subgroups of ${\rm PGL}_2(K)$},
   conference={
      title={Vector bundles and complex geometry},
   },
   book={
      series={Contemp. Math.},
      volume={522},
      publisher={Amer. Math. Soc., Providence, RI},
   },
   date={2010},
   pages={23--29},
   review={\MR{2681719}},
   doi={10.1090/conm/522/10289},
}

\bib{MR1484478}{article}{
    author = {Bosma, Wieb and Cannon, John and Playoust, Catherine},
     title = {The {M}agma algebra system. {I}. {T}he user language},
      note = {Computational algebra and number theory (London, 1993)},
   journal = {J. Symbolic Comput.},
  fjournal = {Journal of Symbolic Computation},
    volume = {24},
      year = {1997},
    number = {3-4},
     pages = {235--265},
      issn = {0747-7171},
   mrclass = {68Q40},
  mrnumber = {MR1484478},
     review={\MR{1484478}},
       doi = {10.1006/jsco.1996.0125},
       URL = {http://dx.doi.org/10.1006/jsco.1996.0125},
}

\bib{MR1994218}{book}{
   author={Bourbaki, Nicolas},
   title={Algebra II. Chapters 4--7},
   series={Elements of Mathematics (Berlin)},
   note={Translated from the 1981 French edition by P. M. Cohn and J. Howie;
   Reprint of the 1990 English edition [Springer, Berlin;  MR1080964
   (91h:00003)]},
   publisher={Springer-Verlag, Berlin},
   date={2003},
   pages={viii+461},
   isbn={3-540-00706-7},
   review={\MR{1994218}},
   doi={10.1007/978-3-642-61698-3},
}

\bibitem[CLY04]{Rademacher}
Chua, Kok Seng and Lang, Mong Lung and Yang, Yifan,
  \newblock \textit{On Rademacher's conjecture: congruence subgroups of genus zero of the modular group},
 \newblock Journal of Algebra,
 \newblock 277,
 \newblock 1,
 \newblock 408--428,
 \newblock 2004,
 \newblock Elsevier


\bib{MR2016709}{article}{
   author={Cummins, C. J.},
   author={Pauli, S.},
   title={Congruence subgroups of ${\rm PSL}(2,{\ZZ})$ of genus less than
   or equal to 24},
   journal={Experiment. Math.},
   volume={12},
   date={2003},
   number={2},
   pages={243--255},
   issn={1058-6458},
   review={\MR{2016709}},
}
available at \href{http://www.uncg.edu/mat/faculty/pauli/congruence/}{http://www.uncg.edu/mat/faculty/pauli/congruence/}

\bib{MR0337993}{article}{
   author={Deligne, P.},
   author={Rapoport, M.},
   title={Les sch\'{e}mas de modules de courbes elliptiques},
   language={French},
   conference={
      title={Modular functions of one variable, II},
      address={Proc. Internat. Summer School, Univ. Antwerp, Antwerp},
      date={1972},
   },
   book={
      publisher={Springer, Berlin},
   },
   date={1973},
   pages={143--316. Lecture Notes in Math., Vol. 349},
   review={\MR{0337993}},
}
		
\bib{MR648603}{book}{
   author={Kubert, Daniel S.},
   author={Lang, Serge},
   title={Modular units},
   series={Grundlehren der Mathematischen Wissenschaften [Fundamental
   Principles of Mathematical Science]},
   volume={244},
   publisher={Springer-Verlag, New York-Berlin},
   date={1981},
   pages={xiii+358},
   isbn={0-387-90517-0},
   review={\MR{648603}},
}


\bib{MR3906177}{article}{
   author={Lombardo, Davide},
   author={Lorenzo Garc\'{\i}a, Elisa},
   title={Computing twists of hyperelliptic curves},
   journal={J. Algebra},
   volume={519},
   date={2019},
   pages={474--490},
   issn={0021-8693},
   review={\MR{3906177}},
   doi={10.1016/j.jalgebra.2018.08.035},
}

\bib{MR0450283}{article}{
   author={Mazur, B.},
   title={Rational points on modular curves},
   conference={
      title={Modular functions of one variable, V},
      address={Proc. Second Internat. Conf., Univ. Bonn, Bonn},
      date={1976},
   },
   book={
      publisher={Springer, Berlin},
   },
   date={1977},
   pages={107--148. Lecture Notes in Math., Vol. 601},
   review={\MR{0450283}},
}
\bib{MR3957898}{article}{
   author={Morrow, Jackson S.},
   title={Composite images of Galois for elliptic curves over $\bold{Q}$ and
   entanglement fields},
   journal={Math. Comp.},
   volume={88},
   date={2019},
   number={319},
   pages={2389--2421},
   issn={0025-5718},
   review={\MR{3957898}},
   doi={10.1090/mcom/3426},
}
\bib{MR3500996}{article}{
    AUTHOR = {Rouse, Jeremy and Zureick-Brown, David},
     TITLE = {Elliptic curves over {$\Bbb Q$} and 2-adic images of {G}alois},
   JOURNAL = {Res. Number Theory},
  FJOURNAL = {Research in Number Theory},
    VOLUME = {1},
      YEAR = {2015},
     PAGES = {Paper No. 12, 34},
      ISSN = {2522-0160},
   MRCLASS = {11G05 (11F80)},
  MRNUMBER = {3500996},
MRREVIEWER = {\'{A}lvaro Lozano-Robledo},
       DOI = {10.1007/s40993-015-0013-7},
       URL = {https://doi.org/10.1007/s40993-015-0013-7},
}		
\bib{MR0387283}{article}{
   author={Serre, Jean-Pierre},
   title={Propri\'{e}t\'{e}s galoisiennes des points d'ordre fini des courbes
   elliptiques},
   language={French},
   journal={Invent. Math.},
   volume={15},
   date={1972},
   number={4},
   pages={259--331},
   issn={0020-9910},
   review={\MR{0387283}},
   doi={10.1007/BF01405086},
}

\bib{MR1466966}{book}{
   author={Serre, Jean-Pierre},
   title={Galois Cohomology},
   note={Translated from the French by Patrick Ion and revised by the
   author},
   publisher={Springer-Verlag, Berlin},
   date={1997},
   pages={x+210},
   isbn={3-540-61990-9},
   review={\MR{1466966}},
   doi={10.1007/978-3-642-59141-9},
}

\bib{MR1291394}{book}{
   author={Shimura, Goro},
   title={Introduction to the arithmetic theory of automorphic functions},
   series={Publications of the Mathematical Society of Japan},
   volume={11},
   note={Reprint of the 1971 original;
   Kan\^{o} Memorial Lectures, 1},
   publisher={Princeton University Press, Princeton, NJ},
   date={1994},
   pages={xiv+271},
   isbn={0-691-08092-5},
   review={\MR{1291394}},
}
 


\bib{MR3671434}{article}{
   author={Sutherland, Andrew V.},
   author={Zywina, David},
   title={Modular curves of prime-power level with infinitely many rational
   points},
   journal={Algebra Number Theory},
   volume={11},
   date={2017},
   number={5},
   pages={1199--1229},
   issn={1937-0652},
   review={\MR{3671434}},
   doi={10.2140/ant.2017.11.1199},}

\bib{MR2721742}{article}{
   author={Zywina, David},
   title={Elliptic curves with maximal Galois action on their torsion
   points},
   journal={Bull. Lond. Math. Soc.},
   volume={42},
   date={2010},
   number={5},
   pages={811--826},
   issn={0024-6093},
   review={\MR{2721742}},
   doi={10.1112/blms/bdq039},
}
\bibitem[Zyw15]{1508.07660}
David Zywina,
\newblock \textit{On the possible images of the mod $l$ representations associated to elliptic curves over $\Q$}, 2015;
\newblock arXiv:1508.07660.




		
\end{biblist}
\end{bibdiv}
\end{document}